\documentclass{qam-l}

\usepackage{amssymb}
\usepackage{accents}
\newtheorem{theorem}{Theorem}[section]

\newtheorem{proposition}[theorem]{Proposition}
\newtheorem{corollary}[theorem]{Corollary}

\theoremstyle{definition}
\newtheorem{definition}[theorem]{Definition}

\theoremstyle{remark}
\newtheorem{remark}[theorem]{Remark}

\numberwithin{equation}{section}

\def\Bbb{\mathbb}

\begin{document}
\title[Minimal Entropy Conditions for General Convex Fluxes]{Minimal Entropy Conditions for  Scalar Conservation Laws with General Convex Fluxes}
%\title{Minimal Entropy Conditions for Scalar Conservation Laws with General Convex Fluxes}

%    Enter the address for every author, even if some are the same.

%    author one information
\author{Gaowei Cao}
\address{Wuhan Institute of Physics and Mathematics, Innovation Academy for Precision Measurement Science and Technology, Chinese Academy of Sciences, Wuhan 430071, China}
\curraddr{(Of the first author) Oxford Centre for Nonlinear Partial Differential Equations, Mathematical Institute, University of Oxford, Oxford, OX2 6GG, UK}
\email{\tt gwcao@apm.ac.cn; caog@maths.ox.ac.uk}
\thanks{Gaowei Cao was supported in part by the National Natural Science Foundation of China No.11701551 and No.11971024, and the China Scholarship Council No.202004910200}

%    author two information
\author{Gui-Qiang G. Chen$^{\dag}$}
\address{Oxford Centre for Nonlinear Partial Differential Equations, Mathematical Institute, University of Oxford, Oxford, OX2 6GG, UK}
%\curraddr{}
\email{\tt chengq@maths.ox.ac.uk}
\thanks{Gui-Qiang G. Chen was supported in part by the UK Engineering and Physical Sciences Research Council Awards EP/L015811/1, EP/V008854, and EP/V051121/1, and the Royal Society-Wolfson Research Merit Award WM090014 (UK)}

%    \subjclass is required.

\date{December 22, 2022 and, in revised form, March 15, 2023.}

\dedicatory{To Costas Dafermos on the Occasion of his 80th Birthday with Admiration and Affection}

\keywords{Entropy solutions, minimal entropy conditions, Radon measure,
convex fluxes, strict convexity, locally Lipschitz, H\"{o}lder continuity, uniqueness,
weak solutions, viscosity solutions, bilinear form, commutator estimates.\\
$^\dag$Corresponding author.}

%    Abstract is required.
\begin{abstract}
We are concerned with the minimal entropy conditions for one-dimensional scalar conservation laws with general convex flux functions.
For such scalar conservation laws, we prove that a single entropy-entropy flux pair $(\eta(u),q(u))$ with $\eta(u)$ of strict convexity is sufficient to single out
an
%the unique
entropy solution from a broad class of weak solutions in $L^\infty_{\rm loc}$ that satisfy the inequality:
$\eta(u)_t+q(u)_x\leq \mu$ in the distributional sense for some non-negative Radon measure $\mu$.
Furthermore, we extend this result to the class of weak solutions in $L^p_{\rm loc}$,
based on the asymptotic behavior of the flux function $f(u)$ and the entropy function $\eta(u)$ at infinity.
The proofs are based on the equivalence between the entropy solutions of one-dimensional scalar conservation laws and the viscosity solutions of the corresponding Hamilton-Jacobi equations, as well as the bilinear form and commutator estimates as employed similarly in the theory of compensated compactness.
\end{abstract}

\maketitle


\section{Introduction}
We are concerned with the minimal entropy conditions for one-dimensional scalar conservation laws with general convex flux functions:
\begin{eqnarray}
&&u_t+f(u)_x=0\qquad\,\, \mbox{for $(t,x)\in {\Bbb R}^+\times{\Bbb R}:=(0,\infty)\times (-\infty,\infty)$},\label{equationu}\\
&&u|_{t=0}=u_0(x), \label{ID}
\end{eqnarray}
for the initial data function $u_0\in L^p_{\rm loc}$ for $p\geq 1$.

It is well known that, due to the nonlinearity of the flux function $f(u)$, no matter how smooth the initial data function $u_0(x)$ is,
the solution may form shock waves generically in a finite time.
Thus, the solution should be understood in a weak sense, which means that the solution as a function with suitable integrability
solves equation \eqref{equationu} in the distributional sense.
In general, the weak solutions are not unique, so we need entropy conditions to characterize the unique entropy solution among the weak solutions.

As shown by Oleinik \cite{OOA}, for equation \eqref{equationu} with uniformly convex flux function $f(u)\in C^2({\Bbb R})$, the entropy condition, so called Condition (E),
is sufficient to single out the unique weak solution (physically relevant) among all possible weak solutions.
Condition (E) is the Oleinik's one-sided inequality for the entropy solution $u(t,x)$:
\begin{equation}\label{conditione}
\frac{u(t,x_2)-u(t,x_1)}{x_2-x_1}\leq\frac{1}{c t}
\qquad\mbox{for\ any\ $x_2>x_1$\ and\ $t>0$},
\end{equation}
where $c:=\inf \{f''(u):u\in {\Bbb R}\}>0$.
Condition $\eqref{conditione}$ implies the regularizing effect that the initial data function in $L^\infty$
are regularized to $BV_{\rm loc}$ instantaneously for the solutions.
This condition also yields many fine properties such as the regularity, the decay rates, and the convergence of approximation schemes, among others, for the solutions; see Dafermos \cite{[DCM]} and Lax \cite{[Lax]}.

For one-dimensional scalar conservation laws with general flux functions, not necessarily convex (even for the multi-dimensional case), a general method for enforcing the uniqueness of solutions in $L^\infty$ was established by Kruzkov \cite{[KSN1]}, by following the earlier results for solutions in $BV_{\rm loc}$ by Conway-Smoller \cite{[CS]} and Vol$'$pert \cite{VA}.
Besides the existence of weak solutions in $L^\infty$, Kruzkov also proved the uniqueness by the so-called Kruzkov's entropy condition:
An entropy solution is a weak solution $u\in L^\infty$ satisfying that
\begin{equation}\label{eta0}
\eta(u)_t+q(u)_x\leq 0\qquad \,\, \mbox{in $\mathcal{D}^\prime$}
\end{equation}
for any $(\eta(u),q(u))\in \{(\eta_k(u),q_k(u))\}_{k\in {\Bbb R}}$, the family of which is defined by
$$
\eta_k(u):=|u-k|, \quad\,\, q_k(u):={\rm sgn}(u-k)\big(f(u)-f(k)\big).
$$
It is equivalent to saying that $\eqref{eta0}$ holds for all entropy-entropy flux pair $(\eta(u),q(u))$ with convex entropy function $\eta(u)$ and
\begin{equation}\label{1.6a}
q'(u)=\eta'(u)f'(u)\qquad\,\, {a.e.}\,\, u \in {\Bbb R}.
\end{equation}

When the flux function is uniformly convex, the two entropy conditions $\eqref{conditione}$ and $\eqref{eta0}$ are equivalent,
which implies that the entropy solutions characterized by Oleinik's condition (E) coincide with the Kruzkov entropy solutions.
In 1989, Kruzkov \cite{[KSN2]} posed an important open question on whether only one single convex entropy $\eta(u)$ satisfying $\eqref{eta0}$
can enforce the uniqueness of the solution, which is called the {\it Minimal~ Entropy~ Condition} in De Lellis-Otto-Westdickenberg \cite{[LOW]}.
In view of the lack of convex entropy functions for hyperbolic systems of conservation laws, the question of {\it Minimal~ Entropy~ Conditions} becomes important for the mathematical theory of hyperbolic conservation laws.

Panov \cite{PEY} first gave a positive answer to this question by proving that the weak solution $u\in L^\infty$ satisfying $\eqref{equationu}$ and $\eqref{eta0}$, with a flux function $f(u)$ and a single entropy function $\eta(u)$ that are both uniformly convex, is the unique entropy solution in Oleinik's sense, or equivalently Kruzkov's sense.
This result was also proved by De Lellis-Otto-Westdickenberg \cite{[LOW]} and Krupa-Vasseur \cite{[KV]}.
In De Lellis-Otto-Westdickenberg \cite{[LOW]}, they further proved that, for the Burgers flux: $f(u)=\frac{1}{2}u^2$ and the special convex entropy $\eta(u)=\frac{1}{2}u^2$,
the weak solution $u\in L^4_{\rm loc}(\Omega)$ satisfying the {\it Minimal~ Entropy~ Condition}:
\begin{equation}\label{etab}
\big(\frac{1}{2} u^2\big)_t+\big(\frac{1}{3} u^3\big)_x\leq \mu\qquad \mbox{in $\mathcal{D}^\prime(\Omega)$},
\end{equation}
for some non-negative Radon measure $\mu$ with $\lim_{r\downarrow 0}\frac{\mu(B_r(t,x))}{r}=0$ for each $(t,x)\in \Omega$, must be the entropy solution of the Burgers equation,
where $B_r(t,x)$ denotes the disk with center $(t,x)$ and radius $r>0$.
For the Kruzkov-type estimates via the entropy inequalities bounded by Radon measures; see Bouchut-Perthame \cite{[BP]}.

One of the motivations of this paper is that the uniform convexity is a strong restriction on both the flux function $f(u)$ and the entropy function $\eta(u)$, since most of convex functions $g(u)$ such as those with $g''(u)=0$ possessing only isolated roots and/or with an asymptotic line, saying $g(u)=|u|^\alpha $ with $\alpha>2$ and $g(u)=e^{ku} $ with $k\neq 0$, are not uniformly convex.
Another motivation is based on the conjecture by De Lellis-Otto-Westdickenberg \cite{[LOW]} that the result for the Burgers equation with the minimal entropy condition \eqref{etab}
could be generalized to allow for different strictly convex flux functions and entropy functions.

The purpose of this paper is to give a positive answer to the question of {\it Minimal Entropy Conditions} by generalizing the previous results in \cite{PEY,[LOW]} in three aspects:
\begin{enumerate}
\item[(i)] The flux function $f(u)$ is required to be only convex, which allows itself to be linear degeneracy in general.

\item[(ii)] The minimal entropy condition similar to $\eqref{etab}$ is generalized to be as follows:
There exists a strictly convex entropy function $\eta(u)$ such that
\begin{equation}\label{etamu}
\eta(u)_t+q(u)_x\leq \mu \qquad \mbox{in $\mathcal{D}^\prime(\Omega)$}
\end{equation}
for some non-negative Radon measure $\mu$ satisfying that
\begin{equation}\label{mu0}
\mathop{\underline{\lim}}\limits_{r \downarrow 0}\frac{1}{r}\mu(B_r(t,x))=0\qquad\mbox{for every $(t,x)\in \Omega$}.
\end{equation}

\item[(iii)]  The weak solution $u$ is required to be only in $L^p_{\rm loc}$ for $p\ge p_0$, instead of $u\in L^\infty$ or $u\in L^4_{\rm loc}$,
with $p_0\ge 1$ determined by the asymptotic behavior of the flux function $f(u)$ and the entropy function $\eta(u)$ as $u\rightarrow \pm \infty$.
\end{enumerate}

As well known in convex analysis ({\it cf.} \cite{[HL]}), a function $g(u)$ is called convex if, for any $u_1,u_2\in {\Bbb R}$ with $u_1\neq u_2$,
$$
g(\lambda u_1+(1-\lambda)u_2)\leq\lambda g(u_1)+(1-\lambda) g(u_2)
\qquad \mbox{for any $\lambda \in [0,1]$},
$$
or equivalent to saying that $g(u)$ is locally Lipschitz continuous and
\begin{equation}\label{convexgf}
g'(u) \,\, \,\,~~{\rm is ~nondecreasing ~on}~ {\Bbb R}.
\end{equation}
Furthermore, a function $g(u)$ is called strictly convex if, for any $u_1,u_2\in {\Bbb R}$ with $u_1\neq u_2$,
$$
g(\lambda u_1+(1-\lambda)u_2)<\lambda g(u_1)+(1-\lambda) g(u_2)
\qquad \mbox{for any $\lambda \in (0,1)$},
$$
or equivalent to saying that $g(u)$ is locally Lipschitz continuous and
\begin{equation}\label{convexg}
g'(u) \,\, \,\,~~{\rm is ~strictly ~increasing ~on}~ {\Bbb R}.
\end{equation}
Notice that the strict convexity condition $\eqref{convexg}$ for the entropy function $\eta(u)$ is equivalent to saying that $\eta(u)$ is convex and satisfies that there is no interval in $\Bbb R$ in which $\eta(u)$ is affine, which is a generalized notion of classical genuine nonlinearity.
In general, such an entropy function $\eta(u)$ allows itself to be not locally uniformly convex, even to possess a measure-zero degenerate set $\{u\,:\, \eta''(u)=0\}$ for the case that $\eta(u)\in C^2({\Bbb R})$.
Furthermore, condition $\eqref{convexg}$ on the entropy function $\eta(u)$ is optimal; otherwise, if $\eta(u)$ is linearly degenerate on some interval $[a,b]$,
then, for $u_0(x)$ ranged in $[a,b]$, inequality $\eqref{etamu}$ does not provide further restrictions on the weak solutions, which can not enforce the uniqueness.
Moreover, the convexity condition $\eqref{convexgf}$ on the flux function $f(u)$ is necessary in general since, for the case of the flux function without convexity,
a single strictly convex entropy function $\eta(u)$ for $\eqref{etamu}$ is not sufficient to enforce the uniqueness; see Remark {\rm $\ref{nonconvex}$}.
In addition, condition $\eqref{mu0}$ on the non-negative Radon measure $\mu$ is optimal; see Remark {\rm $\ref{muoptimal}$} below.

\smallskip
As in Panov \cite{PEY} and De Lellis-Otto-Westdickenberg \cite{[LOW]}, our proofs are also based on the relation between the one-dimensional scalar conservation laws $\eqref{equationu}$ and the corresponding Hamilton-Jacobi equations:
\begin{equation}\label{equationh}
w_t+f(w_x)=0  \qquad\,\,\, \mbox{for $(t,x)\in (0,\infty)\times {\Bbb R}$}.
\end{equation}
Formally, \eqref{equationu} and \eqref{equationh} would be equivalent via the relation: $w_x=u$.
It follows from the existence and uniqueness theory for viscosity solutions of $\eqref{equationh}$, introduced first by Crandall-Lions in \cite{[CL]} (also see Lions \cite{[Lions]}), that the continuous function  $w$ is a viscosity solution of $\eqref{equationh}$ if and only if $u=w_x$ is an entropy solution of $\eqref{equationu}$.
In particular, it can be shown that, if $w$ is a viscosity solution of $\eqref{equationh}$, which can be obtained by the vanishing viscosity method,
then $u=w_x$ is an entropy solution of $\eqref{equationu}$.
In this paper, we prove that the weak solution $u$ satisfying the minimal entropy condition $\eqref{etamu}$ for a single strictly convex entropy function $\eta(u)$
implies that $w$, with $w_x=u$, is actually a viscosity solution of $\eqref{equationh}$,
which yields that $u$ must be an entropy solution satisfying
the entropy inequality $\eqref{eta0}$.
Our analysis is highly motivated by the arguments in De Lellis-Otto-Westdickenberg \cite{[LOW]} (also see Ambrosio-Lecumberry-Rivi\`{e}re \cite{[ALR]});
in particular, the proofs are based on the bilinear form and commutator estimates, for which similar arguments have been employed
in the theory of compensated compactness ({\it cf.} \cite{[CLu],Tartar}).

This paper is organized as follows:
In \S 2, we first introduce some basic concepts and then present the main theorems of this paper.
In \S 3, we prove several lemmas on the averages of functions and properties of linear degeneracy of convex functions for the subsequent development.
In \S 4, we complete the proof of Theorem \ref{Theinfty} for the case that $u\in L^\infty_{\rm loc}$.
In \S 5, we prove Theorem \ref{Thep} for the case that $u\in L^p_{\rm loc}$, based on the asymptotic behavior of the flux function $f(u)$ and the entropy function $\eta(u)$ as $u\to \pm \infty$.

\section{Basic Notions and Main Theorems}
In this section, we first present the notion of entropy solutions of scalar conservation laws \eqref{equationu} and the notion of viscosity solutions of Hamilton-Jacobi equations \eqref{equationh}, respectively.
Then we state the main theorems of this paper and make several related remarks.

\begin{definition}[Entropy Solutions]\label{entropys}
Let $f(u)\in {\rm Lip}_{\rm loc}({\Bbb R})$ and $\Omega \subset{\Bbb R}^+\times{\Bbb R}$.
A function $u\in L^1_{\rm loc}(\Omega)$ is called an entropy solution of the scalar conservation law $\eqref{equationu}$ if, for any convex entropy-entropy flux pair $(\eta(u),q(u))\in {\rm Lip}_{\rm loc}({\Bbb R})$ satisfying $q'(u)=\eta'(u)f'(u)$ almost everywhere, $u$ satisfies $\eqref{equationu}$ in $\mathcal{D}^\prime(\Omega)$ and
\begin{equation}\label{ueta}
\eta(u)_t+q(u)_x\leq 0\qquad \mbox{ in $\mathcal{D}^\prime(\Omega)$}.
\end{equation}
\end{definition}

\smallskip
\begin{definition}[Viscosity Solutions]\label{viscositys}
A function $w\in C(\Omega)$ with $\Omega \subset{\Bbb R}^+\times{\Bbb R}$ is called a viscosity solution of the Hamilton-Jacobi equation $\eqref{equationh}$ if $w$ satisfies that
\begin{enumerate}
\item [\rm (i)] For any $C^1$--function $\varphi$ such that $w-\varphi$ has a local maximum at some point $(t,x)$,
\begin{equation}\label{xi-}
\varphi_t(t,x)+f(\varphi_x(t,x))\leq 0{\rm ;}
\end{equation}

\item [\rm (ii)] For any $C^1$--function $\varphi$ such that $w-\varphi$ has a local minimum at some point $(t,x)$,
\begin{equation}\label{xi+}
\varphi_t(t,x)+f(\varphi_x(t,x))\geq 0.
\end{equation}
\end{enumerate}
A function $w$ is called a viscosity subsolution {\rm (}{\it resp}., supersolution{\rm )}
if {\rm (i)} {\rm (}{\it resp}., {\rm (ii)}{\rm )} holds.
\end{definition}

Our first main theorem is for the case: $u\in L^\infty_{\rm loc}$.

\begin{theorem}[$u\in L^\infty_{\rm loc}$]\label{Theinfty}
Let $f(u)\in {\rm Lip}_{\rm loc}({\Bbb R})$ be a convex flux function of $\eqref{equationu}$.
Assume that $u\in L^\infty_{\rm loc}(\Omega)$ for an open set $\Omega \subset {\Bbb R}^+\times{\Bbb R}$ satisfies
$\eqref{equationu}$ in $\mathcal{D}^\prime(\Omega)$ and
\begin{equation}\label{etamu2}
\eta(u)_t+q(u)_x\leq \mu \qquad \mbox{in $\mathcal{D}^\prime(\Omega)$}
\end{equation}
for some strictly convex entropy function $\eta(u)\in {\rm Lip}_{\rm loc}({\Bbb R})$
as in $\eqref{convexg}$ with some non-negative Radon measure $\mu$ satisfying $\eqref{mu0}$.
Then the locally Lipschitz function $w=w(t,x)$ with $w_x=u$ and $w_t=-f(u)$ is a viscosity solution of the Hamilton-Jacobi
equation $\eqref{equationh}$,
and $u\in L^\infty_{\rm loc}(\Omega)$ is an entropy solution of the scalar
conservation law $\eqref{equationu}$ in $\Omega$.
\end{theorem}

For the case that $u\in L^p_{\rm loc}$ with $p\geq 1$, we need the asymptotic
behavior of the flux function $f(u)$ and the entropy function $\eta(u)$ as $u\to  \pm \infty$.
More precisely, we need the flux function $f(u)$ and the entropy function $\eta(u)$ to satisfy:
\begin{enumerate}
\item[(i)] The flux function $f(u)$ and the entropy function $\eta(u)$ both grow at least linearly as
$u\to \pm\infty$,
$i.e.$,
\begin{equation}\label{growthrate}
\lim_{u \rightarrow \pm\infty}\frac{f(u)}{|u|}\geq M_1>0, \quad
\lim_{u \rightarrow \pm\infty}\frac{\eta(u)}{|u|^{\beta+1}}\geq M_2>0 \qquad {\rm for \ some\ } \beta\geq 0;
\end{equation}

\item[(ii)]  The quadratic form $Q(u):=uq(u)-f(u)\eta(u)$ grows faster than the entropy flux function $q(u):=\int_0^u\eta'(\xi)f'(\xi)\, {\rm d}\xi$, $i.e.$,
\begin{equation}\label{quadratic1}
\lim_{u \rightarrow \pm\infty}\frac{Q(u)}{|q(u)|}=\lim_{u \rightarrow \pm\infty}\frac{uq(u)-f(u)\eta(u)}{|q(u)|}=\infty;
\end{equation}

\item[(iii)]  The entropy function $\eta(u)$ grows not
too slow with respect to
the flux function $f(u)$, $i.e.$, there exist some constants $\gamma\geq 1$ and $C>0$ such that
\begin{equation}\label{quadratic2}
\frac{Q(u)}{\eta(u)}\leq C\Big(\frac{Q(u)}{f(u)}\Big)^\gamma \qquad\,\, \mbox{\rm as $u\rightarrow\pm\infty$}.
\end{equation}
\end{enumerate}

\begin{theorem}[$u\in L^p_{\rm loc}$]\label{Thep}
Let $f(u)\in {\rm Lip}_{\rm loc}({\Bbb R})$ be a convex flux function of $\eqref{equationu}$,
and let $\eta(u)\in {\rm Lip}_{\rm loc}({\Bbb R})$ be a strictly convex entropy function
as in $\eqref{convexg}$ so that
$f(u)$ and $\eta(u)$ satisfy $\eqref{growthrate}$--$\eqref{quadratic2}$.
Assume that  $u\in L^p_{\rm loc}(\Omega)$ such that $Q(u)\in L^1_{\rm loc}(\Omega)$
for an open set $\Omega \subset {\Bbb R}^+\times{\Bbb R}$
satisfies $\eqref{equationu}$ in $\mathcal{D}^\prime(\Omega)$,
and $\eqref{etamu2}$ with $\eqref{mu0}$ in $\mathcal{D}^\prime(\Omega)$.
Then the continuous function $w=w(t,x)$ with $w_x=u$ and $w_t=-f(u)$ is a viscosity solution of
the Hamilton-Jacobi equation $\eqref{equationh}$, which further is  H\"{o}lder continuous in $(t,x)$
for the case that $\beta>0$,
and $u\in L^p_{\rm loc}(\Omega)$ is an
entropy solution of the scalar conservation law
$\eqref{equationu}$ in $\Omega$.
\end{theorem}

\begin{remark}\label{gamma}
Notice that, although condition $\eqref{growthrate}$ admits the flux function $f(u)$
with linear growth rate, $f(u)$ can not be a linear function for large $|u|$,
in which case
condition $\eqref{quadratic1}$ does not hold.
Condition $\eqref{quadratic2}$ means that the entropy function $\eta(u)$, used to enforce the uniqueness,
should be able to control the growth of the flux function $f(u)$ as $u\rightarrow \pm \infty$.
%at infinity.
In particular, under conditions $\eqref{growthrate}$--$\eqref{quadratic1}$, it is direct to check that $\eqref{quadratic2}$ holds for $\gamma=\max\{\lambda, 1\}$ if
\begin{equation}\label{gamma1}
|f'(u)|\leq C|\eta'(u)|^\lambda \qquad\,\, \mbox{\rm as $u\rightarrow\pm\infty$}
\end{equation}
for some $\lambda\geq 0$ and sufficiently large $C>0$. In fact, without loss of generality, consider the case that $u\rightarrow\infty$. Then there exist two cases{\rm:}

\smallskip
${\rm (i)}$ If $\lambda \in [0,1]$, $\eqref{growthrate}$ implies that $f'(u)\lesssim \eta'(u)$ so that $f(u)\lesssim \eta(u)$. Then $\eqref{quadratic2}$ holds for $\gamma=1${\rm;}

\smallskip
${\rm (ii)}$ If $\lambda>1$,
it follows from $\eqref{gamma1}$ that
\begin{equation}\label{gamma2}
f'(u)=f'(u)^{\frac{1}{\lambda}}f'(u)^{\frac{\lambda-1}{\lambda}}
\lesssim \eta'(u)\, f'(u)^{\frac{\lambda-1}{\lambda}}
=\eta'(u)^{\frac{1}{\lambda}}\big(\eta'(u)f'(u)\big)^{\frac{\lambda-1}{\lambda}},
\end{equation}
which, by the H\"{o}lder inequality, implies that
\begin{align}\label{gamma3}
f(u)=\int_0^uf'(\xi)\,{\rm d}\xi
&\lesssim \int_0^u\eta'(\xi)^{\frac{1}{\lambda}}
\big(\eta'(\xi)f'(\xi)\big)^{\frac{\lambda-1}{\lambda}}\,{\rm d}\xi\nonumber\\[1mm]
&\leq \Big(\int_0^u\eta'(\xi)\,{\rm d}\xi\Big)^{\frac{1}{\lambda}}
\Big(\int_0^u\eta'(\xi)f'(\xi)\,{\rm d}\xi\Big)^{\frac{\lambda-1}{\lambda}}\nonumber\\[1mm]
&=(\eta(u)-\eta(0))^{\frac{1}{\lambda}}q(u)^{\frac{\lambda-1}{\lambda}}\nonumber\\[1mm]
&\simeq \eta(u)^{\frac{1}{\lambda}}q(u)^{\frac{\lambda-1}{\lambda}}.
\end{align}
This, by $\eqref{quadratic1}$, infers that
\begin{equation}\label{gamma4}
f(u)^{\lambda}\lesssim \eta(u) q(u)^{\lambda-1}\leq \eta(u) Q(u)^{\lambda-1},
\end{equation}
which implies that $\eqref{quadratic2}$ holds for $\gamma=\lambda$.

As shown above, the reason to require $\gamma\geq 1$ is that, since $\eqref{growthrate}$ and $\eqref{quadratic1}$
imply that $\lim_{u \rightarrow \pm\infty}\frac{Q(u)}{f(u)}=\infty$,
$\eqref{quadratic2}$ also holds for $\gamma=1$ if it holds for some $\gamma\in [0,1)$.
\end{remark}

\begin{remark}\label{growthrateg}
To understand conditions \eqref{quadratic1}--\eqref{quadratic2}
on the asymptotic behavior of the flux function $f(u)$ and the entropy function $\eta(u)$, we present the examples that both $f(u)$ and $\eta(u)$ possess the polynomial growth rates, $i.e.$,
for some constants $\accentset{\sim}{M}_1,\accentset{\sim}{M}_2 \geq 0$ and $\alpha,\beta \geq 0$,
\begin{equation}\label{growthrate+}
f'(u)\, \simeq\, \accentset{\sim}{M}_1{\rm sgn}(u)|u|^{\alpha}, \quad \eta'(u)\, \simeq\, \accentset{\sim}{M}_2{\rm sgn}(u)|u|^{\beta} \qquad \mbox{as $u\rightarrow \pm\infty$}.
\end{equation}
\begin{enumerate}
\item[\rm (i)]  For the case that $\alpha>0 $ and $\beta>0$ in $\eqref{growthrate+}$, it is direct to check that
\begin{equation}\label{quadratic4}
|q(u)|\, \sim \,
|u|^{\alpha+\beta+1}, \quad Q(u) \,\sim \,
|u|^{\alpha+\beta+2}
\qquad \mbox{as $u\rightarrow\pm\infty$},
\end{equation}
so that $\eqref{quadratic1}$ holds, and $\eqref{quadratic2}$ holds for $\gamma=\max\{\frac{\alpha+1}{\beta+1},1\}$ and sufficiently large $C>0$.

\smallskip
\item[\rm (ii)]  For the case that $\alpha=0$ and $\beta>0$ in $\eqref{growthrate+}$, consider
the convex flux function $f(u)$ with more detailed asymptotic properties{\rm :} As an example, for some
constants $m_1>0$ and $\tilde{\alpha} \in (0,1)$,
\begin{equation}\label{quadratic6}
f'(u)\, \simeq\, \accentset{\sim}{M}_1 {\rm sgn}(u)(1-m_1|u|^{-\tilde{\alpha}}) \qquad \mbox{\rm as $u \rightarrow \pm\infty$}.
\end{equation}
Then it is direct to check that
\begin{equation}\label{quadratic7}
|q(u)|\,\sim\,
|u|^{\beta+1}, \quad Q(u)\,\sim\,
|u|^{\beta+2-\tilde{\alpha}}
\qquad \mbox{as $u\rightarrow\pm\infty$},
\end{equation}
so that $\eqref{quadratic1}$ holds, and $\eqref{quadratic2}$ holds for $\gamma=1$ and sufficiently large $C>0$.

Furthermore, if $\tilde{\alpha}=1$, then $Q(u)\sim|u|^{\beta+1}\log|u|\sim |q(u)|\log|u|$
so that $\eqref{quadratic1}$ holds, and $\eqref{quadratic2}$ holds for $\gamma=1$ and sufficiently large $C>0$.

\smallskip
\item[\rm (iii)]   For the case that $\alpha>0$ and $\beta=0$ in $\eqref{growthrate+}$,
consider the strictly convex entropy function $\eta(u)$ with more detailed asymptotic properties{\rm :}
As an example, for some constants
$m_2>0$ and $\tilde{\beta} \in (0,1)$,
\begin{equation}\label{quadratic9}
\eta'(u)\,\simeq\, \accentset{\sim}{M}_2{\rm sgn}(u)(1-m_2|u|^{-\tilde{\beta}}) \qquad \mbox{\rm as $u \rightarrow \pm\infty$}.
\end{equation}
Then
%it is direct to check that
\begin{equation}\label{quadratic10}
|q(u)|\,\sim\,
|u|^{\alpha+1}, \quad Q(u)\,\sim\,
|u|^{\alpha+2-\tilde{\beta}}
\qquad \mbox{as $u\rightarrow\pm\infty$},
\end{equation}
so that  $\eqref{quadratic1}$ holds, and $\eqref{quadratic2}$ holds for $\gamma=\frac{\alpha+1-\tilde{\beta}}{1-\tilde{\beta}}$ and sufficiently large $C>0$.

\item[\rm (iv)]   For the case that $\alpha=0$ and $\beta=0$ in $\eqref{growthrate+}$,
as an example, suppose $\eqref{quadratic6}$ and $\eqref{quadratic9}$ hold for some constants
$\tilde{\alpha}, \tilde{\beta} \in (0,1)$ with $\tilde{\alpha}+\tilde{\beta}<1$.
Then
%it is direct to check that
\begin{equation}\label{quadratic13}
|q(u)|\,\sim\, |u|, \quad Q(u)\,\sim\,
|u|^{2-\tilde{\alpha}-\tilde{\beta}}
\qquad \mbox{as $u\rightarrow\pm\infty$},
\end{equation}
so that  $\eqref{quadratic1}$ holds, and $\eqref{quadratic2}$ holds for $\gamma=1$ and sufficiently large $C>0$.

Furthermore, if $\tilde{\alpha}+\tilde{\beta}=1$, then $Q(u)\sim |u|\log|u|\sim |q(u)|\log|u|$ so that $\eqref{quadratic1}$ holds, and $\eqref{quadratic2}$  holds for $\gamma= 1$ and sufficiently large $C>0$.
\end{enumerate}
\end{remark}

\begin{remark}\label{phyentropy}
Suppose that the flux function $f(u)$ and the entropy function $\eta(u)$ satisfies that, for some constants
$\accentset{\sim}{M}_1, \accentset{\sim}{M}_2>0$ and $\alpha>0$,
\begin{equation}\label{quadratic90}
f'(u)\,\simeq\, \accentset{\sim}{M}_1{\rm sgn}(u)|u|^\alpha, \quad
 \eta'(u)\,\simeq\, \accentset{\sim}{M}_2{\rm sgn}(u)(\log|u|+1)
\qquad {\rm as\ } u \rightarrow \pm\infty.
\end{equation}
Then it is direct to check that
\begin{equation}\label{quadratic100}
|q(u)|\,\sim\,
|u|^{\alpha+1}\log|u|, \quad Q(u)\,\sim\,
|u|^{\alpha+2}
\qquad \mbox{as $u\rightarrow\pm\infty$},
\end{equation}
so that  $\eqref{quadratic1}$ holds, and $\eqref{quadratic2}$
holds for $\gamma=\alpha+1$ and sufficiently large $C>0$.

For the case that $\alpha=0$, suppose that the flux function $f(u)$ satisfies $\eqref{quadratic6}$ with $\tilde{\alpha}\in(0,1)$, and the entropy function $\eta(u)$ satisfies $\eqref{quadratic90}$.
Then it is direct to check that
\begin{equation}\label{quadratic101}
|q(u)|\,\sim\,
|u| \log|u|, \quad Q(u)\,\sim\,
|u|^{2-\tilde{\alpha}}
\qquad\,\, \mbox{as $u\rightarrow\pm\infty$},
\end{equation}
so that  $\eqref{quadratic1}$ holds, and $\eqref{quadratic2}$
holds for $\gamma= 1$ and sufficiently large $C>0$.
\end{remark}

\begin{remark}\label{fastg}
If the flux function $f(u)$ grows too fast with respect to
the entropy function
$\eta(u)$, $\eqref{quadratic2}$ may not hold. For example,
suppose that, for some constants $\accentset{\sim}{M}_1, \accentset{\sim}{M}_2>0$
and $\beta\geq 0$,
\begin{equation}\label{fast1}
f'(u)\,\simeq\, \accentset{\sim}{M}_1 {\rm sgn}(u)e^{|u|}, \quad
\eta'(u)\,\simeq\, \accentset{\sim}{M}_2{\rm sgn}(u)|u|^\beta
\qquad\,\, \mbox{\rm as $u \rightarrow \pm\infty$}.
\end{equation}
Then it is direct to check that
\begin{equation}\label{fast2}
|q(u)|\,\sim\,   |u|^{\beta}e^{|u|}, \quad Q(u)\,\sim\, |u|^{\beta+1}e^{|u|}
\qquad \mbox{as $u\rightarrow\pm\infty$},
\end{equation}
so that  $\eqref{quadratic1}$ holds.
 Owing to $\frac{Q(u)}{\eta(u)}\sim e^{|u|}$ and $\frac{Q(u)}{f(u)}\sim |u|^{\beta+1}$, $\eqref{quadratic2}$ does not hold for any $\gamma\geq 1$.
\end{remark}

\begin{remark}\label{sip}
In Theorem $\ref{Thep}$, we assume that  $u\in L^p_{\rm loc}(\Omega)$
such that $Q(u)\in L^1_{\rm loc}(\Omega)$.
For the case in $\eqref{quadratic4}$, we need to require $u\in L^p_{\rm loc}(\Omega)$
with $p\geq \alpha+\beta+2$, which are similar for the cases in
$\eqref{quadratic7}$, $\eqref{quadratic10}$, $\eqref{quadratic13}$, $\eqref{quadratic100}$, and $\eqref{quadratic101}$.
For the case that $Q(u)\sim|u|^{\beta+1}\log|u|$ with $\beta\geq 0$  in Remark {\rm \ref{growthrateg}(ii)} for $\tilde{\alpha}=1$,
it suffices to require only that $|u|^{\beta+1}\log|u|\in L^1_{{\rm loc}}(\Omega)$.
\end{remark}

\begin{remark}\label{muoptimal}
Condition $\eqref{mu0}$ on the non-negative Radon measure is optimal{\rm ;} otherwise, $\eqref{etamu2}$ can not enforce the uniqueness if there exists
$c_0>0$ such that, for small $r>0$,
\begin{equation}\label{muc0}
\mu(B_r(\bar{t},\bar{x}))\geq c_0r \qquad \mbox{for some points $(\bar{t},\bar{x})\in \Omega$}.
\end{equation}
In fact, for the Riemann problem of $\eqref{equationu}$ for
$u_0(x)=u_\pm$ with $u_+>u_-$ for $\pm(x-x_0)>0$, $\eqref{etamu2}$ with $\eqref{muc0}$
admits two weak solutions{\rm :} one is clearly a rarefaction wave,
and the other is an under-compressive shock $\mathcal{S}$ {\rm(}passing through $\Omega${\rm )} defined by
\begin{equation}\label{unshock}
 u(t,x)=\begin{cases}
u_-\quad\mbox{if $x<x_0+s_0t$},\\[1mm]
u_+\quad\mbox{if $x>x_0+s_0 t$},
\end{cases}
\end{equation}
for sufficiently small $u_+-u_->0$  with $f'(u_+{-0})>f'(u_-{+0})$,
where the shock speed $s_0$ is determined by $s_0=\frac{[f(u)]_\pm}{[u]_\pm}$.

In order to show that the under-compressive shock $u(t,x)$ in $\eqref{unshock}$
satisfies $\eqref{etamu2}$ with $\eqref{muc0}$,
it suffices to check
that $\mu_\eta:=\eta(u)_t+q(u)_x\leq\mu$ holds for points $(\bar{t},\bar{x})\in \Omega \cap \mathcal{S}$. Let $u_+-u_->0$ be sufficiently small such that
\begin{equation}\label{uuc0}
0< \int_{u_-}^{u_+}\eta'(\xi)\big(f'(\xi)-s_0\big)\,{\rm d}\xi\leq \frac{c_0}{2},
\end{equation}
where, by $\eqref{calP}$ later, the left-hand inequality always holds for $u_+>u_-$ with $f'(u_+{-0})>f'(u_-{+0})$.
According to $\eqref{muc0}$--$\eqref{uuc0}$, for sufficiently small $r>0$, we obtain, as desired,
\begin{align*}
\mu_\eta(B_r(\bar{t},\bar{x}))
&= \int_{B_r(\bar{t},\bar{x})}\big(\eta(u)_t+q(u)_x\big)\,{\rm d}t{\rm d}x\\
&=2r
\big(s_0(\eta(u_-)-\eta(u_+))-(q(u_-)-q(u_+))\big) \\
&= 2r\int_{u_-}^{u_+}\eta'(\xi)\big(f'(\xi)-s_0\big)\,{\rm d}\xi
\leq c_0r\leq \mu(B_r(\bar{t},\bar{x})).
\end{align*}
\end{remark}

\begin{remark}\label{nonconvex}
 Our proofs depend highly on the convexity of the flux functions $f(u)$,
 $i.e.$, $f'(u)$ is nondecreasing.
 As shown in Dafermos {\rm \cite{[DCM1]}},
both the regularity and large-time behavior of entropy solutions of
the scalar conservation laws without convexity are highly related to
the quantity of $f'(u(t,x))$, instead of the quantity of solution $u(t,x)$ itself.
For the case of the flux function $f(u)$ without convexity, a single strictly convex entropy function $\eta(u)$
may not be sufficient to enforce the uniqueness{\rm ;} instead of which the minimal number
of strictly convex entropy functions $\eta(u)$ needed to enforce the uniqueness
may depend on the number of inflection points of the flux function $f(u)$,
or equivalently the number of the maximum intervals, on each of which function $f'(u)$,
the first derivative of the flux function $f(u)$, is monotone.
\end{remark}

As mentioned above,
our proofs are based on the equivalence between the entropy solutions
of the one-dimensional scalar conservation laws $\eqref{equationu}$
and the viscosity solutions of the corresponding Hamilton-Jacobi equations $\eqref{equationh}$.

In fact, this can be seen via the vanishing viscosity method:
Let $w(t,x)$ be the unique viscosity solution of $\eqref{equationh}$
with the Cauchy initial data:
\begin{equation}\label{ID-HJ}
w|_{t=0}=w_0(x).
\end{equation}
Then
it can be proved that $w(t,x)$ can be regarded as the limit function of
the viscosity approximate solution sequence $w^\varepsilon$ when $\varepsilon\to 0$,
as proved in Crandall-Lions \cite{[CL]},
where $w^\varepsilon$ is the unique solution of the Cauchy problem:
\begin{equation}\label{equationhe}
w^\varepsilon_t+f(w^\varepsilon_x)=\varepsilon w^\varepsilon_{xx}
\end{equation}
with the Cauchy initial data $\eqref{ID-HJ}$
for each fixed $\varepsilon>0$.
Furthermore,
$v^\varepsilon :=w^\varepsilon_x$ solves
\begin{equation}\label{equationve}
v^\varepsilon_t+f(v^\varepsilon)_x=\varepsilon v^\varepsilon_{xx},
\end{equation}
and $v^\varepsilon \rightarrow v=w_x$, where $v$ is the unique entropy solution
of $\eqref{equationu}$ with $v(0,x)=w_0'(x)$.

On the other hand, if a function $\bar{w}$ defined by $\bar{w}_x=u$ and $\bar{w}_t=-f(u)$
as in Theorem $\ref{Theinfty}$ ($resp.,$ Theorem $\ref{Thep}$)
is a viscosity solution of $\eqref{equationh}$ and $\eqref{ID-HJ}$, by the uniqueness of viscosity solutions, we conclude that
$$
w=\bar{w}, \qquad w_x\mathop{=}\limits^{a.e.}\bar{w}_x.
$$
For more details, see Crandall-Evans-Lions \cite{[CEL]}
for the case of bounded $w_0(x)$ and Ishii \cite{[HI]} for the case of unbounded $w_0(x)$.

Therefore, function $u$ in Theorem $\ref{Theinfty}$ ($resp.,$ Theorem  $\ref{Thep}$) satisfies that
$$
u=\bar{w}_x\mathop{=}\limits^{a.e.}w_x=v,
$$
which means that this function $u$
is the unique entropy solution of the scalar conservation law $\eqref{equationu}$.

\medskip
In a similar way, if $u(t,x)$ is the unique entropy solution of the scalar
conservation law $\eqref{equationu}$ with $u_0(x)=w'_0(x)$,
then we
can prove that function $w(t,x)$ defined by $w_x=u$ and $w_t=-f(u)$ satisfying $\eqref{ID-HJ}$ is the viscosity solution of $\eqref{equationh}$ and $\eqref{ID-HJ}$.

\begin{remark}\label{entropysolution}
According to Cao-Chen-Yang {\rm \cite{[CCY]}},
for a strictly convex flux function $f(u)$ of $\eqref{equationu}$ as in $\eqref{convexg}$,
$w=w(t,x)$ defined by $w_x=u$ and $w_t=-f(u)$ is Lipschitz continuous.
Furthermore, $w^\pm_x(t,x)=u^\pm(t,x)$ are well-defined pointwise with $u^+(t,x)=u^-(t,x)$
almost everywhere and
\begin{equation}\label{traces}
w^+_x(t,x)=u^+(t,x)\leq u^-(t,x)=w^-_x(t,x).
\end{equation}
%Notice that
As shown in Corollary $\ref{Oleinikentropy}$ below,
$\eqref{traces}$ is equivalent to the following shock admissibility condition{\rm :}
$$
s\big(\eta(u^+)-\eta(u^-)\big)-\big(q(u^+)-q(u^-)\big)\geq 0
$$
for any convex entropy-entropy pair $(\eta(u),q(u))${\rm ;} see also {\rm Dafermos \cite{[DCM]}}.
This implies that $u=w_x$ is actually the entropy solution of the scalar conservation law
$\eqref{equationu}$.
\end{remark}

To complete the proof of Theorems $\ref{Theinfty}$ ($resp.,$ Theorem  $\ref{Thep}$), it suffices
to prove that $w$ defined by $w_x=u$ and $w_t=-f(u)$ as in Theorem $\ref{Theinfty}$ ($resp.,$ Theorem $\ref{Thep}$)
is a viscosity solution of  $\eqref{equationh}$
for $u\in L^\infty_{\rm loc}$ ($resp.,$ $u\in L^p_{\rm loc}$).

\section{Averages of Functions and Properties of Linear Degeneracy for Convex Functions}
In order to prove the main theorems,
we need Proposition $\ref{ProB}$ below on the averages of functions,
which generalizes Proposition 3.2 in De Lellis-Otto-Westdickenberg \cite{[LOW]},
and Proposition $\ref{ProCon}$ on the properties of linear degeneracy for general convex functions.

\begin{definition}[Averages of Functions]\label{average}
Assume that $\mu$ is a probability measure on ${\Bbb R}$.
For every vector-valued map $h\in L^1({\Bbb R},\mu)$, set
\begin{equation}\label{averagev}
\langle h(u)\rangle:=\int_\Omega h(u)\,{\rm d}\mu(u).
\end{equation}
Let $f(u),\eta(u)\in W^{1,\infty}_{\rm loc}({\Bbb R})$
and $q(u):=\int^u_0\eta'(\xi)f'(\xi)\,{\rm d}\xi$.
When $\mu$ is compactly supported, we define the bilinear form{\rm :}
\begin{equation}\label{defB}
\begin{split}
\mathcal{B}(f,\eta)&:=\langle (-f(u),u)\cdot (\eta(u),q(u))\rangle-\langle (-f(u),u)\rangle\cdot \langle (\eta(u),q(u))\rangle\\
&\,=\langle uq(u)\rangle-\langle u\rangle \langle q(u) \rangle
- \big(\langle \eta(u) f(u)\rangle -\langle\eta(u)\rangle\langle f(u)\rangle\big).
\end{split}
\end{equation}
When $\mu$ is of noncompact support, we define $\mathcal{B}(f,\eta)$ whenever the
functions in $\eqref{defB}$ are all $\mu-$summable.
\end{definition}

\smallskip
Then we have the following properties of the bilinear form $\mathcal{B}(f,\eta)$ in $\eqref{defB}$:

\begin{proposition}[Bilinear Form]\label{ProB}
Let $f(u),\eta(u)\in W^{1,\infty}_{\rm loc}({\Bbb R})$ and $q(u):=\int^u_0\eta'(\xi)f'(\xi)\,{\rm d}\xi$.
\begin{enumerate}
\item [\rm (i)] If $\bar{f}(u)=f(u)-(au+b)$ and $\bar{\eta}(u)=\eta(u)-(cu+d)$
for constants $a, b, c$, and $d$, then
\begin{equation}\label{linearB}
\mathcal{B}(f,\eta)=\mathcal{B}(\bar{f},\eta)=\mathcal{B}(f,\bar{\eta})= \mathcal{B}(\bar{f},\bar{\eta}).
\end{equation}

\item [\rm (ii)] If $f(u)$ and $\eta(u)$ are both convex,
then
\begin{equation}\label{convexB}
\mathcal{B}(f,\eta)
\geq \langle \eta(u)-\eta (\langle u\rangle ) \rangle \langle f(u)-f(\langle u\rangle ) \rangle
\geq 0.
\end{equation}
\end{enumerate}
\end{proposition}

\begin{proof} We now give the proof for the two cases respectively.

\smallskip
(i) Since $\bar{f}(u)=f(u)-(au+b)$,
\begin{equation}\label{qbar}
\bar{q}(u)=\int^u_0\eta'(\xi)\bar{f}'(\xi)\,{\rm d}\xi=\int^u_0\eta'(\xi)\big(f'(\xi)-a\big)\,{\rm d}\xi=q(u)-a\eta(u)+a\eta(0).
\end{equation}
Then, by the definition of $\mathcal{B}(f,\eta)$ in $\eqref{defB}$,
we have
\begin{align*}
\mathcal{B}(f,\eta)&=\mathcal{B}(\bar{f}(u)+au+b,\eta (u))\\[1mm]
&=\langle u(\bar{q}(u)+a\eta(u)-a\eta(0))\rangle -\langle u\rangle \langle \bar{q}(u)+a\eta(u)-a\eta(0)\rangle
 \\[1mm]
&\quad - \langle\eta(u)(\bar{f}(u)+au+b)\rangle
  + \langle \bar{f}(u)+au+b\rangle
  \langle \eta(u)\rangle
  \\[1mm]
&=\langle u\bar{q}(u)\rangle  -\langle u\rangle\langle \bar{q}(u)\rangle- \langle \eta(u)\bar{f}(u)\rangle +\langle\bar{f}(u)\rangle \langle \eta(u)\rangle
=\,\mathcal{B}(\bar{f},\eta).
\end{align*}
Similarly, for $\bar{\eta}(u)=\eta(u)-(cu+d)$, we have
$$
\mathcal{B}(f,\eta)=\mathcal{B}(f,\bar{\eta})=\mathcal{B}(\bar{f},\bar{\eta}).
$$

\smallskip
(ii) By the definition of $\mathcal{B}(f,\eta)$ in $\eqref{defB}$, we have
\begin{align}\label{calB}
\mathcal{B}(f,\eta)&=\big\langle (u- \langle u\rangle)q(u) \big\rangle-\big\langle\big(\eta(u)- \langle \eta(u)\rangle\big)\big(f(u)-f(\langle u\rangle)\big)\big\rangle\nonumber\\[1mm]
&=\big\langle (u- \langle u\rangle)q(u) \big\rangle -\big\langle \big(\eta(u)-\eta(\langle u\rangle)\big)\big(f(u)-f( \langle u\rangle)\big)\big\rangle\nonumber\\[1mm]
&\qquad +\big\langle \big(\langle\eta(u) \rangle -\eta(\langle u\rangle)\big)\big(f(u)-f(\langle u\rangle)\big)\big\rangle\nonumber\\[1mm]
&=\big\langle (u- \langle u\rangle)q(u)-\big(\eta(u)-\eta(\langle u\rangle)\big)\big(f(u)-f( \langle u\rangle)\big)\big\rangle\nonumber\\[1mm]
&\qquad +\big\langle\eta(u) -\eta(\langle u\rangle)\big\rangle \big\langle f(u)-f(\langle u\rangle)\big\rangle\nonumber\\[1mm]
&=: \langle P(u,\langle u\rangle)\rangle+\langle Q(u,\langle u\rangle)\rangle,
\end{align}

For $\langle P(u,\langle u\rangle)\rangle$, $P(u,\langle u\rangle)=0$ when $u=\langle u\rangle$.
When $u\neq\langle u\rangle$,  by the nondecreasing of $f'(u)$,
there exists at least
one $\zeta$ such that
$$
c(\zeta):=\frac{f(u)-f(\langle u\rangle)}{u-\langle u\rangle}\in [f'(\zeta{-0}),f'(\zeta{+0})].
$$
Since  $f'(u)$ and $\eta'(u)$ are both nondecreasing,
we see that,
if $u> \langle u\rangle$, then
\begin{align}\label{calP}
P(u,\langle u\rangle)&= (u-\langle u\rangle )\int^u_{ \langle u\rangle }\eta'(\xi)\big(f'(\xi)-c(\zeta)\big)\,{\rm d}\xi \nonumber\\
&=(u- \langle u\rangle ) \Big(\int^u_\zeta\eta'(\xi)\big(f'(\xi)-c(\zeta)\big)\,{\rm d}\xi - \int^\zeta_{\langle u\rangle}\eta'(\xi)\big(c(\zeta)-f'(\xi)\big)\,{\rm d}\xi\Big)\nonumber\\
&\geq (u- \langle u\rangle ) \Big(\int^u_\zeta d(\zeta)\big(f'(\xi)-c(\zeta)\big)\,{\rm d}\xi- \int^\zeta_{\langle u\rangle}d(\zeta)\big(c(\zeta)-f'(\xi)\big)\,{\rm d}\xi\Big)\nonumber\\
&=(u- \langle u\rangle ) d(\zeta)\int^u_{\langle u\rangle} \big(f'(\xi)-c(\zeta)\big)\,{\rm d}\xi\nonumber\\
&=0,
\end{align}
where $d(\zeta)\in [\eta'(\zeta{-0}),\eta'(\zeta{+0})]$. Similarly, if $u< \langle u\rangle$, then
\begin{equation}\label{equationP00}
P(u,\langle u\rangle)=\big(\langle u\rangle-u\big) \int^{\langle u\rangle}_u\eta'(\xi)\big(f'(\xi)-c(\zeta)\big)\,{\rm d}\xi\geq 0.
\end{equation}
Thus, we conclude that $P(u,\langle u\rangle)\geq 0$ so that
\begin{equation}\label{equationP}
\langle P(u,\langle u\rangle)\rangle\geq 0.
\end{equation}

For $\langle Q(u,\langle u\rangle)\rangle$, since $f(u)$ and $\eta(u)$ are both convex, we use the Jensen inequality
to see that
\begin{equation}\label{equationQ}
\langle f(u)\rangle -f(\langle u\rangle )\geq 0, \qquad \langle \eta(u)\rangle -\eta(\langle u\rangle )\geq 0.
\end{equation}

Combining $\eqref{calB}$--$\eqref{equationQ}$ together, we conclude that
$$
\mathcal{B}(f,\eta)=\langle P(u,\langle u\rangle)\rangle+\langle Q(u,\langle u\rangle)\rangle\geq \langle \eta(u)-\eta(\langle u\rangle)\rangle\langle f(u)-f(\langle u\rangle)\rangle\geq 0,
$$
as desired.
\end{proof}

\begin{remark}\label{uniformlyconvex}
If $f(u),\eta(u)\in C^2({\Bbb R})$ are both uniformly convex with $c_1:=\inf f''(u)>0$ and $c_2:=\inf \eta''(u)>0$, in view of $\eqref{linearB}$, choose $\bar{f}(u)$ and $\bar{\eta}(u)$ such that
$\bar{f}(\langle u\rangle)=\bar{f}'(\langle u\rangle )=0$ and $\bar{\eta}(\langle u\rangle )=\bar{\eta}'(\langle u\rangle)=0$. By a simple calculation, we have
\begin{equation}\label{convexBuniform}
\mathcal{B}(f,\eta)=\mathcal{B}(\bar{f},\bar{\eta})\geq\langle \bar{\eta}(u) \rangle \langle \bar{f}(u) \rangle\geq \frac{c_1c_2}{4}\langle |u-\langle u\rangle |^2 \rangle^2\geq \frac{c_1c_2}{4}\langle |u-\langle u\rangle|\rangle^4.
\end{equation}
\end{remark}

Finally, we give some properties of a general convex function $f(u)\in {\rm Lip}_{{\rm loc}}({\Bbb R})$.

\begin{proposition}[Linear Degeneracy]\label{ProCon}
Let $f(u)\in {\rm Lip}_{{\rm loc}}({\Bbb R})$ be convex,
and let the point sets $I^{\pm}(u)$ be defined by
\begin{equation}\label{degenI}
I^{\pm}(u):=\big\{v\in \bar{{\Bbb R}}\,:\, f(v)-f(u)-f'(u{\pm 0})(v-u)=0\big\}
\qquad\,\, \mbox{\rm for any $u\in {\Bbb R}$},
\end{equation}
where $\bar{{\Bbb R}}:=[-\infty,\infty]$. Then
\begin{enumerate}
\item [\rm (i)] $I^{\pm}(u)$ is a single point set or a closed interval in $\bar{{\Bbb R}}$ given by
\begin{equation}\label{degenIc}
I^{\pm}(u)=[\inf I^{\pm}(u),\ \sup I^{\pm}(u) ].
\end{equation}
In particular, if $f'(u)$ is strictly increasing, then $I^\pm(u)=\{u\}$ for any $u\in {\Bbb R}$.

\item [\rm (ii)] The first derivative of function $f(u)$ satisfies that
\begin{equation}\label{degenIp}
\begin{cases}
f'(v)<f'(u{\pm 0}) \qquad {\rm for \ } v< \inf I^{\pm}(u),\\[1mm]
f'(v)=f'(u{\pm 0}) \qquad {\rm for \ } v \in I^{\pm}(u),\\[1mm]
f'(v)>f'(u{\pm 0}) \qquad {\rm for \ } v> \sup I^{\pm}(u).
\end{cases}
\end{equation}
Furthermore, if $I^{\pm}(u)$ is a closed interval,  then $f(u)$
is linearly degenerate on $I^{\pm}(u)$.
\end{enumerate}
\end{proposition}

\begin{proof} We prove the two properties in two steps, respectively.

\smallskip
(i)\ First, it follows from $\eqref{degenI}$ that $u\in I^{\pm}(u)$, $i.e.$, $I^\pm(u)\neq \emptyset$.

If $I^\pm(u)$ has more than one point, then, for any two points $u_1,u_2\in I^\pm(u)$ with $u_1<u_2$,
$[u_1,u_2]\subset I^\pm(u)$. In fact, by the definition of $I^\pm(u)$ in $\eqref{degenI}$, we have
\begin{equation}\label{degenId}
f(u_i)-f(u)-f'(u{\pm 0})(u_i-u)=0 \qquad\,\, \mbox{for $i=1,2$}.
\end{equation}
Since $f(u)$ is convex, for any $v\in (u_1,u_2)$, {\it i.e.},
$v=\theta u_1+(1-\theta)u_2$ for $\theta \in (0,1)$,
\begin{equation}\label{degenIdc+}
f(v)\geq f(u)+f'(u{\pm 0})(v-u).
\end{equation}
On the other hand, from $\eqref{degenId}$, we have
\begin{equation}\label{degenIdc-}
f(v)=f(\theta u_1+(1-\theta)u_2)\leq \theta f(u_1)+(1-\theta)f(u_2)=f(u)+f'(u{\pm 0})(v-u).
\end{equation}
Combining $\eqref{degenIdc+}$ with $\eqref{degenIdc-}$, $v\in I^\pm(u)$ so that
$[u_1,u_2]\subset I^\pm(u)$.
Thus, $I^\pm(u)$ is an interval. By the continuity of function $f(u)$, $I^\pm(u)$
is a closed interval in $\bar{{\Bbb R}}$ given by $\eqref{degenIc}$.

For the case that $f'(u)$ is strictly increasing, the definition of $I^\pm(u)$ implies that $I^\pm(u)=\{u\}$ for any $u\in {\Bbb R}$.

\smallskip
(ii)\ If $v\in I^\pm(u)$, the definition of $I^\pm(u)$ and $\eqref{degenIc}$ imply that $f'(v)= f'(u{\pm 0})$.
On the other hand, $u\in I^\pm(u)$ and, if $f'(v)= f'(u{\pm 0})$ with $v\neq u$, then $f'(\xi)$
is constant for $\xi$ lying between $v$ and $u$, which
implies that $v\in I^\pm(u)$ by the definition of $I^\pm(u)$ and (i).
We conclude that $f'(v)= f'(u{\pm 0})$
if and only if  $v\in I^\pm(u)$. Furthermore, by the nondecreasing of $f'(v)$, $f'(v)\leq f'(u{\pm 0})$  for $v<u$ and $f'(v)\geq f'(u{\pm 0})$ for $v>u$. Hence, $\eqref{degenIp}$ is true.

Furthermore, if $I^{\pm}(u)$ is a closed interval, the definition of $I^{\pm}(u)$ implies that the convex function $f(u)$ is linearly degenerate on $I^{\pm}(u)$.
\end{proof}

\begin{corollary}\label{Oleinikentropy}
Let $f(u)$ and $\eta(u)$ be both convex and locally Lipschitz.
%Let $f(u)\in {\rm Lip}_{\rm loc}({\Bbb R})$ be a convex,
%and let $\eta(u)\in {\rm Lip}_{\rm loc}({\Bbb R})$ be a strictly convex
%as in $\eqref{convexg}$.
Denote $q(u):=\int_0^u\eta'(\xi)f'(\xi){\rm d}\xi$. Then,
for any $u,v\in{\Bbb R}$,
\begin{equation}\label{degenIcnonlinear}
P(v,u):=(v-u)\big(s(\eta(u)-\eta(v))-(q(u)-q(v))\big)\geq 0
\qquad \mbox{for $s=\frac{f(u)-f(v)}{u-v}$}.
\end{equation}
Moreover, if $\eta(u)$ is strictly convex as in $\eqref{convexg}$, then
\begin{equation}\label{degenIO}
P(v,u)=0 \quad {\rm for}\ v \in I^-(u)\cup I^+(u), \qquad  P(v,u)>0
\quad {\rm for}\ v\notin I^-(u)\cup I^+(u).
\end{equation}
\end{corollary}

\begin{proof}
Inequality $\eqref{degenIcnonlinear}$ follows by replacing $u$ and $\langle u \rangle$ in $\eqref{calP}$ by $v$ and $u$, respectively.

We now assume that $\eta(u)$ is strictly convex as in $\eqref{convexg}$.
If $v\in I^{\pm}(u)$, it follows from $\eqref{degenI}$ that
 $s=\frac{f(u)-f(v)}{u-v}=f'(u\pm0)$ and from $\eqref{degenIp}$
 that $f'(\xi)=f'(u\pm 0)=s$ for any $\xi$ lying between $u$ and $v$.
 Then
$$
P(v,u)=(v-u)\int^v_u\eta'(\xi)\big(f'(\xi)-s\big)\,{\rm d}\xi=0.
$$

For $v\notin I^-(u)\cup I^+(u)$, by $\eqref{degenIp}$,
$\pm(f'(v)-f'(u\pm 0))>0$ if $\pm(v-u)>0$,
which implies that $P(v,u)>0$ by using $\eqref{calP}$ with strictly increasing $\eta'(u)$.
\end{proof}

\begin{remark}\label{conshock}
Along a discontinuity $\mathcal{S}$ of a weak solution $u=u(t,x)$ of $\eqref{equationu}$ with the left and right traces $u^-:=u(t,x-0)$ and $u^+:=u(t,x+0)$, respectively,
$$
s[u]-[f(u)]:=s(u^+-u^-)-\big(f(u^+)-f(u^-)\big)=0.
$$
By
Corollary {\rm $\ref{Oleinikentropy}$}, if $u^-\in I^-(u^+)\cup I^+(u^+)$, then $s\big(\eta(u^+)-\eta(u^-)\big)-\big(q(u^+)-q(u^-)\big)=0$ so that the discontinuity $\mathcal{S}$ is a contact discontinuity. If $u^-\notin I^-(u^+)\cup I^+(u^+)$, then
\begin{equation}\label{equvilance}
s\big(\eta(u^+)-\eta(u^-)\big)-\big(q(u^+)-q(u^-)\big)> 0
\end{equation}
if and only if
\begin{equation}\label{equivalence-2}
f'(u^-{-0})>f'(u^+{+0}),
\end{equation}
which is exactly the Lax entropy condition {\rm (see \cite{[Lax]})}.
\end{remark}

\begin{remark}\label{generalconvex}
Since $I^\pm(u)$ is a single point set or an interval,
for the case that $\pm\infty\in I^\pm(u)$,
the equality{\rm :}
$$
f(\pm\infty)-f(u)-f'(u{\pm 0})(\pm\infty-u)=0
$$
in the definition of $I^\pm(u)$ in $\eqref{degenI}$ is well-defined by regarding it
as the limit of $f(v)-f(u)-f'(u{\pm 0})(v-u)\equiv 0$
for $v\in I^\pm(u)\rightarrow \pm\infty$.
\end{remark}

\section{Proof of Theorem $\ref{Theinfty}$}
 In Theorem $\ref{Theinfty}$, since $u\in L^\infty_{\rm loc}$ is assumed,
 then $w$ is a local Lipschitz function.
 We now prove that $w$ is both a viscosity subsolution and supersolution.
 Therefore, by Definition $\ref{viscositys}$, $w$ is a viscosity solution.

\subsection{Viscosity subsolution}
By the definition of $w$ in Theorem $\ref{Theinfty}$,
\begin{equation}\label{equationht}
w_t=-f(u)=-f(w_x) \qquad \mbox{{\it a.e.} in $\Omega$}.
\end{equation}
Let $\zeta\in C^\infty_{\rm c}({\Bbb R}^+\times{\Bbb R})$ be non-negative and satisfy
$\int_{{\Bbb R}^+\times{\Bbb R}}\zeta(t,x)\,{\rm d}t{\rm d}x=1$, and set
$$
\zeta_\varepsilon(t,x)=\frac{1}{\varepsilon^2}\zeta(\frac{t}{\varepsilon},\frac{x}{\varepsilon}).
$$
Since $f(u)$ is convex, by the Jensen inequality, we see that
$$
0=\big(w_t+f(w_x)\big)*\zeta_\varepsilon\geq w_t* \zeta_\varepsilon+f(w_x*\zeta_\varepsilon)=(w*\zeta_\varepsilon)_t+f((w*\zeta_\varepsilon)_x).
$$
Hence, $w*\zeta_\varepsilon$ is a classical subsolution so that it is also a viscosity
subsolution (see Crandall-Lions \cite[Corollary I.6]{[CL]}).
Since $w$ is continuous, $w*\zeta_\varepsilon$ converges, locally uniform,
to $w$ as $\varepsilon$ tends to $0$.
Thus, $w$ is also a viscosity subsolution, by the stability result
in Crandall-Lions \cite[Theorem I.2]{[CL]}.

\subsection{Viscosity supersolution}
We divide the proof into six steps.

\smallskip
1. To prove that $w$ is a viscosity supersolution,
we need to show that, if $\varphi$ is a $C^1$--function such that $w-\varphi$
has a minimum at some point $(t,x)\in \Omega$, then $\big(\varphi_t+f(\varphi_x)\big)(t,x)\geq 0$.

Without loss of generality, we may assume that  $(t,x)=(0,0)$ and $(w-\varphi)(0,0)=0$.
Then it suffices to show
\begin{equation}\label{equationxi}
\big(\varphi_t+f(\varphi_x)\big)(0,0)=0.
\end{equation}

\smallskip
2. To simplify the notation, we use $p(\varepsilon,\delta)\lesssim k(\varepsilon,\delta)$
to denote that there exist a large constant $C>0$ and a small constant $c>0$ such that
\begin{equation}\label{deflesssim}
p(\varepsilon,\delta)\leq C k(\varepsilon,\delta)\qquad  \mbox{for $|(\varepsilon, \delta)|\leq c$}.
\end{equation}

Let $\varphi_\varepsilon=\varphi-\varepsilon|(t,x)|$ with $\varepsilon\in(0,1]$.
Then $w-\varphi_\varepsilon$ has a strict minimum at $(0,0)$.
We define $\Omega_{\varepsilon,\delta}$ for $\delta >0$ as
\begin{equation}\label{defOmega}
\Omega_{\varepsilon,\delta}:=\big\{(t,x)\,:\, (w-\varphi_\varepsilon)(t,x)<\delta\big\}.
\end{equation}
Since $w$ is continuous and $w-\varphi$
has a strict minimum at $(0,0)$,  we see that $\Omega_{\varepsilon,\delta}$ is an open set
and, for any $(t,x)\in \Omega_{\varepsilon,\delta}$,
$$
(w-\varphi_\varepsilon)(t,x)
=(w-\varphi)(t,x)+\varepsilon|(t,x)|\geq(w-\varphi)(0,0)+\varepsilon|(t,x)|=\varepsilon|(t,x)|,
$$
which infers from $\eqref{defOmega}$ that
\begin{equation}\label{OmegaB}
\Omega_{\varepsilon,\delta}\subset\big\{(t,x)\,:\,\varepsilon|(t,x)|<\delta\big\}
=B_{r}(0,0)\qquad\,\, \mbox{for $r:=\frac{\delta}{\varepsilon}$}.
\end{equation}

Denote $\langle\, \cdot\,\rangle_{\varepsilon,\delta}$ by
\begin{equation}\label{integralu}
\langle u\rangle_{\varepsilon,\delta}
:=\def\avint{\mathop{\,\rlap{--}\!\!\int}\nolimits} \avint_{\Omega_{\varepsilon,\delta}} u(t,x)\,{\rm d}t{\rm d}x
=\frac{1}{|\Omega_{\varepsilon,\delta}|}\int_{\Omega_{\varepsilon,\delta}}u(t,x)\,{\rm d}t{\rm d}x.
\end{equation}
Then we observe that
\begin{equation}\label{xifu}
\langle ((\varphi_{\varepsilon})_t, (\varphi_{\varepsilon})_x)\rangle_{\varepsilon,\delta}
=\langle (-f(u), u)\rangle_{\varepsilon,\delta}.
\end{equation}
Indeed, this can be seen from the definition of $w$ in Theorem $\ref{Theinfty}$ that
\begin{equation}\label{xifuh}
\left\langle (-f(u),u)\right\rangle_{\varepsilon,\delta}-\left\langle ((\varphi_{\varepsilon})_t, (\varphi_{\varepsilon})_x)\right\rangle_{\varepsilon,\delta}
=\left \langle \big((w-\varphi_\varepsilon)_t,(w-\varphi_\varepsilon)_x\big)
\right\rangle_{\varepsilon,\delta}.
\end{equation}
From \eqref{defOmega}--\eqref{OmegaB},
$\min\{w-\varphi_\varepsilon-\delta,0\}$ is continuous with compact support so that
\begin{equation}\label{hxi}
\left \langle (w-\varphi_\varepsilon)_t\right\rangle_{\varepsilon,\delta}
=\frac{1}{|\Omega_{\varepsilon,\delta}|}\int_{\Omega_{\varepsilon,\delta}}(w-\varphi_\varepsilon)_t\,{\rm d}t{\rm d}x
=\frac{1}{|\Omega_{\varepsilon,\delta} |}\int_{{\Bbb R}^2}\big(\min\{w-\varphi_\varepsilon-\delta,0\}\big)_t\,{\rm d}t{\rm d}x
=0.
\end{equation}
Similarly, we obtain that $\langle (w-\varphi_\varepsilon)_x\rangle_{\varepsilon,\delta}=0$.
This infers that $\eqref{xifu}$ holds.

\medskip
3. We now prove $\eqref{equationxi}$.
To achieve this, we will first show in Step 4 below that
\begin{equation}\label{caletafg0}
0\leq \langle \eta(u)-\eta(\langle u\rangle_{\varepsilon,\delta})\rangle_{\varepsilon,\delta}\langle f(u)-f(\langle
u\rangle_{\varepsilon,\delta})\rangle_{\varepsilon,\delta} \leq \mathcal{B}_{\varepsilon,\delta}(f,\eta)\lesssim \varrho(\varepsilon,\delta),
\end{equation}
where $\mathcal{B}_{\varepsilon,\delta}(f,\eta)$ is defined
via replacing $\langle\,\cdot\,\rangle$ in $\eqref{defB}$ by
$\langle\,\cdot\,\rangle_{\varepsilon,\delta}$ in $\eqref{integralu}$,
and $\varrho(\varepsilon,\delta)$ is given by
\begin{equation}\label{calBpg0}
\varrho(\varepsilon,\delta):=\frac{1}{\delta}\mu(B_{\frac{\delta}{\varepsilon}}(0,0))+\frac{\delta}{\varepsilon}+\varepsilon.
\end{equation}
From $\eqref{OmegaB}$, $\delta=r\varepsilon$. Then $\eqref{calBpg0}$ implies
\begin{equation}\label{calBpg0k}
\lim_{\varepsilon\rightarrow 0} \lim_{r_k\rightarrow 0}\varrho(\varepsilon,r_k\varepsilon)=\lim_{\varepsilon\rightarrow 0} \lim_{r_k\rightarrow 0}\Big(\frac{1}{r_k}\mu(B_{r_k}(0,0))\varepsilon^{-1}+r_k+\varepsilon\Big)=0,
\end{equation}
where $\{r_k\}$ is a subsequence of $r$ in $\eqref{mu0}$ such that
\begin{equation}\label{mu0k}
\lim\limits_{r_k \downarrow 0}\frac{1}{r_k}\mu(B_{r_k}(0,0))=0.
\end{equation}

Then, with $\eqref{caletafg0}$--$\eqref{mu0k}$, we will show in Step 5 for the case that $f'(u)$ is strictly increasing and in Step 6 for the case that $f'(u)$ is nondecreasing below that
\begin{equation}\label{calff00}
\lim_{\varepsilon\rightarrow 0} \lim_{r_k\rightarrow 0}|\langle f(u)-f(\langle u\rangle_{\varepsilon,r_k\varepsilon})\rangle_{\varepsilon,r_k\varepsilon}|=0.
\end{equation}

Combining  $\eqref{OmegaB}$ and $\eqref{xifu}$ with $\eqref{caletafg0}$--$\eqref{calff00}$ together, we have
\begin{align}\label{calxiinfty}
\big|\big(\varphi_t+f(\varphi_x)\big)(0,0)\big|&\lesssim\big|\langle \varphi_t\rangle_{\varepsilon,r_k\varepsilon}+ f(\langle\varphi_x\rangle_{\varepsilon,r_k\varepsilon})\big|+r_k\nonumber\\[1mm]
&\lesssim\big|\langle (\varphi_{\varepsilon})_t\rangle_{\varepsilon,r_k\varepsilon}
  + f(\langle(\varphi_{\varepsilon})_x\rangle_{\varepsilon,r_k\varepsilon})\big|+\varepsilon+r_k\nonumber\\[1mm]
&=\big|-\langle f(u)\rangle_{\varepsilon,r_k\varepsilon}+ f(\langle u\rangle_{\varepsilon,r_k\varepsilon})\big|+\varepsilon+r_k\nonumber\\[1mm]
&=\big|\langle f(u)-f(\langle u\rangle_{\varepsilon,r_k\varepsilon})\rangle_{\varepsilon,r_k\varepsilon}\big|+\varepsilon+r_k.
\end{align}
Using $\eqref{calff00}$,
by letting first $r_k$ and then $\varepsilon$ go to $0$ in $\eqref{calxiinfty}$, we conclude $\eqref{equationxi}$ as desired, which means that $w$ is a viscosity supersolution.

\medskip
4. We now prove $\eqref{caletafg0}$ with $\varrho(\varepsilon,\delta)$ given by $\eqref{calBpg0}$. We first notice that
\begin{equation}\label{deltaomega}
\delta^2 \lesssim |\Omega_{\varepsilon,\delta}| \qquad\mbox{for\ small $\delta>0$}.
\end{equation}
Since $u\in L^\infty_{\rm loc}$, and $\varphi$ is smooth,
we use $\eqref{OmegaB}$ and $\eqref{xifu}$ to see that
\begin{equation}\label{ufuxi}
\langle |u|\rangle_{\varepsilon,\delta}\lesssim 1,\quad
\langle|f(u)|\rangle_{\varepsilon,\delta}\lesssim 1, \qquad
\langle|(\varphi_{\varepsilon})_t|\rangle_{\varepsilon,\delta}\lesssim 1, \quad
\langle|(\varphi_{\varepsilon})_x|\rangle_{\varepsilon,\delta}\lesssim 1.
\end{equation}
It follows from the definition of $w$ in Theorem $\ref{Theinfty}$ that
\begin{equation}\label{hthx}
\langle|w_t|\rangle_{\varepsilon,\delta}\lesssim 1, \qquad
\langle|w_x|\rangle_{\varepsilon,\delta}\lesssim 1.
\end{equation}
Then, from $\eqref{ufuxi}$--$\eqref{hthx}$, we obtain
\begin{equation}\label{hxithxix}
\textstyle\left\langle \left| \big((w-\varphi_\varepsilon)_t,(w-\varphi_\varepsilon)_x\big)\right| \right\rangle_{\varepsilon,\delta}
\leq \left\langle \left| (w_t,w_x) \right|\right\rangle_{\varepsilon,\delta}
+ \left\langle \left|((\varphi_{\varepsilon})_t, (\varphi_{\varepsilon})_x)
\right|\right\rangle_{\varepsilon,\delta} \lesssim 1.
\end{equation}

From $\eqref{hxithxix}$, we have
\begin{align}\label{minhxidelta}
&\int_{{\Bbb R}^2} \left| \big((\min\{w-\varphi_\varepsilon -\delta,0\})_t,(\min\{w-\varphi_\varepsilon -\delta,0\})_x \big)\right|\,{\rm d}t{\rm d}x\nonumber\\[1mm]
&=\int_{\Omega_{\varepsilon,\delta}}\left|\big((w-\varphi_\varepsilon)_t ,(w-\varphi_\varepsilon)_x\big)
\right|\,{\rm d}t{\rm d}x
\lesssim |\Omega_{\varepsilon,\delta}|.
\end{align}
Then, by the Sobolev inequality,
\begin{equation}\label{minhxideltaomega}
\Big(\int_{{\Bbb R}^2} (\min\{w-\varphi_\varepsilon-\delta,0\})^2\,{\rm d}t{\rm d}x\Big)^{\frac{1}{2}}
\lesssim |\Omega_{\varepsilon,\delta}|,
\end{equation}
which, due to the H\"{o}lder inequality, implies
\begin{equation}\label{minhxideltaomega32}
 -\int_{{\Bbb R}^2} \min\{w-\varphi_\varepsilon -\delta,0\}\,{\rm d}t{\rm d}x
\leq
|\Omega_{\varepsilon,\delta}|^{\frac{1}{2}}
\Big(\int_{{\Bbb R}^2}(\min\{w-\varphi_\varepsilon -\delta,0\})^2\,{\rm d}t{\rm d}x\Big)^{\frac{1}{2}}
\lesssim|\Omega_{\varepsilon,\delta}|^{\frac{3}{2}}.
\end{equation}

For $\delta>0$, we denote
\begin{equation}\label{Idelta}
I(\delta):=-\int_{{\Bbb R}^2}\min\{w-\varphi_\varepsilon-\delta,0\}\,{\rm d}t{\rm d}x
=\int^\delta_0 |\Omega_{\varepsilon,\sigma}\,|{\rm d}\sigma.
\end{equation}
Then, from $\eqref{minhxideltaomega32}$, we obtain the following differential inequality:
\begin{equation}\label{Ideltainequality}
I(\delta)\lesssim |\Omega_{\varepsilon,\delta}|^{\frac{3}{2}}
=\Big(\frac{{\rm d}}{{\rm d}\delta}I(\delta)\Big)^{\frac{3}{2}}.
\end{equation}
Since $I(\delta)>0$ for $\delta>0$, it follows from $\eqref{Ideltainequality}$
that
$
1\lesssim\frac{{\rm d}}{{\rm d}\delta}\big(I(\delta)^{\frac{1}{3}}\big),
$
which infers that
$\delta^3\lesssim I(\delta).$
Noticing that $|\Omega_{\varepsilon,\delta}|$ is a non-decreasing function of $\delta$, we have
$$
\delta^2\lesssim \frac{1}{\delta}I(\delta)=\frac{1}{\delta}\int^\delta_0|\Omega_{\varepsilon,\sigma}|\,{\rm d}\sigma\leq|\Omega_{\varepsilon,\delta}|.
$$

We now estimate $\mathcal{B}_{\varepsilon,\delta}(f,\eta)$ in $\eqref{caletafg0}$.
In fact, from $\eqref{defB}$ and $\eqref{xifu}$--$\eqref{xifuh}$, we have
\begin{align}\label{calBinfty}
\mathcal{B}_{\varepsilon,\delta}(f,\eta)&=\left\langle(-f(u),u) \cdot(\eta(u),q(u))\right\rangle_{\varepsilon,\delta}-\left\langle
(-f(u),u) \right\rangle_{\varepsilon,\delta}\cdot\left\langle(\eta(u),q(u))\right\rangle_{\varepsilon,\delta}\nonumber\\[1mm]
&=\left\langle (-f(u),u)\cdot (\eta(u),q(u))\right\rangle_{\varepsilon,\delta}
-\left\langle((\varphi_{\varepsilon})_t, (\varphi_{\varepsilon})_x)
\cdot(\eta(u),q(u)) \right\rangle_{\varepsilon,\delta}\nonumber\\[1mm]
&\quad\,+\left\langle ((\varphi_{\varepsilon})_t, (\varphi_{\varepsilon})_x)
\cdot(\eta(u),q(u)) \right\rangle_{\varepsilon,\delta}
-\left\langle ((\varphi_{\varepsilon})_t, (\varphi_{\varepsilon})_x)\right\rangle_{\varepsilon,\delta}
\cdot
\left\langle (\eta(u),q(u))\right\rangle_{\varepsilon,\delta}\nonumber\\[1mm]
&\leq\left\langle\big((w-\varphi_\varepsilon)_t ,(w-\varphi_\varepsilon)_x \big)
\cdot(\eta(u),q(u))\right\rangle_{\varepsilon,\delta}\nonumber\\[1mm]
&\quad\,+\sup_{\Omega_{\varepsilon,\delta}}\big|((\varphi_{\varepsilon})_t,(\varphi_{\varepsilon})_x)
  -\left\langle((\varphi_{\varepsilon})_t, (\varphi_{\varepsilon})_x)\right\rangle_{\varepsilon,\delta}\big|\left\langle \left|
(\eta(u),q(u))\right|\right\rangle_{\varepsilon,\delta}\nonumber\\[1mm]
&=:I_1(\varepsilon,\delta)+I_2(\varepsilon,\delta).
\end{align}
For $I_1(\varepsilon,\delta)$,
it follows from $\eqref{etamu2}$, $\eqref{OmegaB}$, and $\eqref{deltaomega}$ that
\begin{align}\label{calIinfty}
I_1(\varepsilon,\delta)&=\frac{1}{|\Omega_{\varepsilon,\delta}|}\int_{\Omega_{\varepsilon,\delta}}
\big((w-\varphi_\varepsilon )_t ,(w- \varphi_\varepsilon)_ x\big)
\cdot(\eta(u),q(u))\,{\rm d}t{\rm d}x\nonumber\\
&=\frac{1}{|\Omega_{\varepsilon,\delta}|}\int_{{\Bbb R}^2}\big((\min\{w-\varphi_\varepsilon -\delta,0\})_t ,(\min\{w- \varphi_\varepsilon -\delta,0\})_ x\big)
\cdot(\eta(u),q(u))\,{\rm d}t{\rm d}x\nonumber\\
&=\frac{1}{|\Omega_{\varepsilon,\delta}|}\int_{{\Bbb R}^2} \big({-}\min\{w-\varphi_\varepsilon -\delta,0\}\big) \big(\eta(u)_t+q(u)_x\big)\,{\rm d}t{\rm d}x\nonumber\\
&\leq \frac{1}{|\Omega_{\varepsilon,\delta}|}\int_{{\Bbb R}^2} \max\{\delta-(w-\varphi_\varepsilon),0\}\,{\rm d}\mu\nonumber\\
&\leq \frac{\delta}{|\Omega_{\varepsilon,\delta}|}\mu(\Omega_{\varepsilon,\delta})\leq \frac{\delta}{|\Omega_{\varepsilon,\delta}|}\mu(B_{\frac{\delta}{\varepsilon}}(0,0))
\lesssim\frac{1}{\delta}\mu(B_{\frac{\delta}{\varepsilon}}(0,0)).
\end{align}
For $I_2(\varepsilon,\delta)$, it follows from $u\in L^\infty_{\rm loc}$
that $\langle|\eta(u)|\rangle_{\varepsilon,\delta}\lesssim 1$
and $\langle|q(u)|\rangle_{\varepsilon,\delta}\lesssim 1$.
Then, according to the definition of $I_2(\varepsilon,\delta)$ in $\eqref{calBinfty}$, we have
\begin{align}\label{calIIinfty}
I_2(\varepsilon,\delta)
&\textstyle\lesssim \sup\limits_{\Omega_{\varepsilon,\delta}}\big|((\varphi_{\varepsilon})_t, (\varphi_{\varepsilon})_x)
 -\left\langle ((\varphi_{\varepsilon})_t, (\varphi_{\varepsilon})_x)\right\rangle_{\varepsilon,\delta}\big|\nonumber\\
&\textstyle\leq \mathop{\rm osc}\limits_{\Omega_{\varepsilon,\delta}}((\varphi_{\varepsilon})_t, (\varphi_{\varepsilon})_x)
\leq \mathop{\rm osc}\limits_{\Omega_{\varepsilon,\delta}}(\varphi_t , \varphi_x)+2\varepsilon
\displaystyle\lesssim\frac{\delta}{\varepsilon}+\varepsilon.
\end{align}

According to $\eqref{calBinfty}$--$\eqref{calIIinfty}$, we have
\begin{equation}\label{caluinfty}
\mathcal{B}_{\varepsilon,\delta}(f,\eta)
\lesssim\frac{1}{\delta}\mu(B_{\frac{\delta}{\varepsilon}}(0,0))+\frac{\delta}{\varepsilon}+\varepsilon,
\end{equation}
which, by combining with $\eqref{convexB}$, implies $\eqref{caletafg0}$ with $\varrho(\varepsilon,\delta)$
given by $\eqref{calBpg0}$.

\medskip
5. Now we first prove $\eqref{calff00}$ for the case that $f'(u)$ is strictly increasing.
According to Proposition \ref{ProB} and the convexity of $f(u)$ and $\eta(u)$,
if $\bar{f}(u)$ and $\bar{\eta}(u)$ are chosen as
\begin{equation}\label{barfeta}
\begin{cases}
\bar{f}(u):=f(u)-f(\langle u\rangle_{\varepsilon,\delta})
 -f'(\langle u\rangle_{\varepsilon,\delta}{+0})\big(u-\langle u\rangle_{\varepsilon,\delta}\big)\geq 0,\\[1mm]
\bar{\eta}(u):=\eta(u)-\eta(\langle u\rangle_{\varepsilon,\delta})
-\eta'(\langle u\rangle_{\varepsilon,\delta}{+0})\big(u-\langle u\rangle_{\varepsilon,\delta}\big)\geq 0,
\end{cases}
\end{equation}
then, from $\eqref{linearB}$--$\eqref{convexB}$ and $\eqref{caletafg0}$, we have
\begin{equation}\label{convexBuniform0}
0\leq\langle \bar{\eta}(u) \rangle_{\varepsilon,\delta} \langle \bar{f}(u) \rangle_{\varepsilon,\delta}
\leq \mathcal{B}_{\varepsilon,\delta}(\bar{f},\bar{\eta})=\mathcal{B}_{\varepsilon,\delta}(f,\eta)\lesssim \varrho(\varepsilon,\delta).
\end{equation}

On the other hand, notice that $\Omega_{\varepsilon,\delta}\subset B_{\frac{\delta}{\varepsilon}}(0,0)$
is uniformly bounded in $(\varepsilon,\delta)$ with $\delta\leq\varepsilon$.
Since $f(u)\in {\rm Lip}_{\rm loc}({\Bbb R})$  and $u\in L^\infty_{\rm loc}$,
then, for $\delta\leq\varepsilon$, there exists a constant $C>0$
independent of $\varepsilon$ and $\delta$ such that
\begin{equation}\label{calfuinfty}
0\leq\langle \bar{f}(u)\rangle_{\varepsilon,\delta}=\langle f(u)-f(\langle u\rangle_{\varepsilon,\delta})\rangle_{\varepsilon,\delta}
\leq C\left\langle |u-\langle u\rangle_{\varepsilon,\delta}|\right\rangle_{\varepsilon,\delta}.
\end{equation}
This means that, in order to prove $\eqref{calff00}$, it suffices to show
\begin{equation}\label{caluu00}
\lim_{\varepsilon\rightarrow 0} \lim_{r_k\rightarrow 0}\,
\langle |u-\langle u\rangle_{\varepsilon,r_k\varepsilon}|\rangle_{\varepsilon,r_k\varepsilon}=0.
\end{equation}
In fact, for any fixed $\sigma>0$, we have
\begin{align}\label{calulessigma}
&\frac{1}{|\Omega_{\varepsilon,r_k\varepsilon}|}\int_{\Omega_{\varepsilon,r_k\varepsilon}\bigcap\{|u-\langle u\rangle_{\varepsilon,r_k\varepsilon}|\leq\sigma\}}|u-\langle u\rangle_{\varepsilon,r_k\varepsilon}|\,{\rm d}t{\rm d}x\nonumber\\
&\leq\frac{1}{|\Omega_{\varepsilon,r_k\varepsilon}|}\int_{\Omega_{\varepsilon,r_k\varepsilon}\bigcap\{|u-\langle u\rangle_{\varepsilon,r_k\varepsilon}|\leq\sigma\}}\sigma\,{\rm d}t{\rm d}x\leq\sigma.
\end{align}
Thus, if we can prove that, for any fixed $\sigma>0$,
\begin{equation}\label{calulimit}
\mathop{\overline{\lim}}\limits_{\varepsilon\rightarrow 0}\mathop{\overline{\lim}}
\limits_{r_k\rightarrow 0}\frac{1}{|\Omega_{\varepsilon,r_k\varepsilon}|}\int_{\Omega_{\varepsilon,r_k\varepsilon}\bigcap\{|u-\langle u\rangle_{\varepsilon,r_k\varepsilon}|>\sigma\}}|u-\langle u\rangle_{\varepsilon,r_k\varepsilon}|\,{\rm d}t{\rm d}x=0,
\end{equation}
then it follows from
$\eqref{calulessigma}$--$\eqref{calulimit}$ that,
for any fixed $\sigma>0$,
$$
\mathop{\overline{\lim}}\limits_{\varepsilon\rightarrow 0} \mathop{\overline{\lim}}\limits_{r_k\rightarrow 0}\,\langle |u-\langle u\rangle_{\varepsilon,r_k\varepsilon}|\rangle_{\varepsilon,r_k\varepsilon}\leq \sigma,
$$
which yields $\eqref{caluu00}$ directly.

\smallskip
We now prove $\eqref{calulimit}$ by contradiction.
If
$\eqref{calulimit}$ does not hold, then there exist $\sigma_0>0$ and $c_0>0$ such that
\begin{equation}\label{calulimit0}
\mathop{\overline{\lim}}\limits_{\varepsilon\rightarrow 0}\mathop{\overline{\lim}}\limits_{r_k\rightarrow 0}\frac{1}{|\Omega_{\varepsilon,r_k\varepsilon}|}\int_{\Omega_{\varepsilon,r_k\varepsilon}\bigcap\{|u-\langle u\rangle_{\varepsilon,r_k\varepsilon}|>\sigma_0\}}\big|u-\langle u\rangle_{\varepsilon,r_k\varepsilon}\big|\,{\rm d}t{\rm d}x\geq 2c_0>0,
\end{equation}
which is equivalent to saying that there exists a subsequence $(\varepsilon_j,r_{k_i})$ such that,
for small $\varepsilon_j$ and $r_{k_i}$,
\begin{equation}\label{calulimit000}
\frac{1}{\big|\Omega_{\varepsilon_j,r_{k_i}\varepsilon_j}\big|}
\int_{\Omega_{\varepsilon_j,r_{k_i}\varepsilon_j}\bigcap
\{|u-\langle u\rangle_{\varepsilon_j,r_{k_i}\varepsilon_j}|>\sigma_0\}}
\big|u-\langle u\rangle_{\varepsilon_j,r_{k_i}\varepsilon_j}\big|\,{\rm d}t{\rm d}x\geq c_0>0.
\end{equation}

Since $f'(u)$ and $\eta'(u)$ are both strictly increasing, we have
\begin{equation}\label{genfeta}
\begin{cases}
(u-\langle u\rangle_{\varepsilon,\delta})\big(f'(u)-f'(\langle u\rangle_{\varepsilon,\delta}{+0})\big)>0 \qquad &{\rm for}\ u\neq \langle u\rangle_{\varepsilon,\delta},\\[1mm]
(u-\langle u\rangle_{\varepsilon,\delta})\big(\eta'(u)-\eta'(\langle u\rangle_{\varepsilon,\delta}{+0})\big)>0 \qquad &{\rm for}\ u\neq \langle u\rangle_{\varepsilon,\delta}.
\end{cases}
\end{equation}
Then, from the definition of $\bar{f}(u)$ in $\eqref{barfeta}$, for $|u-\langle u\rangle_{\varepsilon,\delta}|>\sigma_0$,
\begin{align}\label{ftalyorcc}
\bar{f}(u)&=f(u)-f(\langle u\rangle_{\varepsilon,\delta})
-f'(\langle u\rangle_{\varepsilon,\delta}{+0})(u-\langle u\rangle_{\varepsilon,\delta})\nonumber\\[1mm]
&=\int_0^1 \big|f'(\langle u\rangle_{\varepsilon,\delta}+\theta(u-\langle u\rangle_{\varepsilon,\delta}))-f'(\langle u\rangle_{\varepsilon,\delta}{+0})
\big|{\rm d}\theta\,|u-\langle u\rangle_{\varepsilon,\delta}|\nonumber\\[1mm]
&\geq\int_0^1 \big|f'(\langle u\rangle_{\varepsilon,\delta}\pm\theta\sigma_0)
-f'(\langle u\rangle_{\varepsilon,\delta}{+0})\big|{\rm d}\theta\,|u-\langle u\rangle_{\varepsilon,\delta}|.
\end{align}
Hence, from $\eqref{calulimit000}$, we have
\begin{align}\label{ftalyor0}
\langle \bar{f}(u)\rangle_{\varepsilon_j,r_{k_i}\varepsilon_j}
&\geq\frac{1}{|\Omega_{\varepsilon_j,r_{k_i}\varepsilon_j}|}\int_{\Omega_{\varepsilon_j,r_{k_i}\varepsilon_j}\bigcap\{|u-\langle u\rangle_{\varepsilon_j,r_{k_i}\varepsilon_j}|>\sigma_0\}}\bar{f}(u)\,{\rm d}t{\rm d}x\nonumber\\
&\geq\min\Big\{\int_0^1 \big|f'(\langle u\rangle_{\varepsilon_j,r_{k_i}\varepsilon_j}\pm\theta\sigma_0)
 -f'(\langle u\rangle_{\varepsilon_j,r_{k_i}\varepsilon_j}{+0})\big|{\rm d}\theta\Big\}\nonumber\\
&\qquad \times\frac{1}{|\Omega_{\varepsilon_j,r_{k_i}\varepsilon_j}|}\int_{\Omega_{\varepsilon_j,r_{k_i}\varepsilon_j}\bigcap\{|u-\langle u\rangle_{\varepsilon_j,r_{k_i}\varepsilon_j}|>\sigma_0\}}\big|u-\langle u\rangle_{\varepsilon_j,r_{k_i}\varepsilon_j}\big|\,{\rm d}t{\rm d}x\nonumber\\
&\geq c_0\min\Big\{\int_0^1 \big|f'(\langle u\rangle_{\varepsilon_j,r_{k_i}\varepsilon_j}\pm\theta\sigma_0)
-f'(\langle u\rangle_{\varepsilon_j,r_{k_i}\varepsilon_j}{+0})\big|{\rm d}\theta\Big\}\nonumber\\
&=:P_f(\langle u\rangle_{\varepsilon_j,r_{k_i}\varepsilon_j},\sigma_0,c_0).
\end{align}
Similarly, for $\bar{\eta}(u)$ as in $\eqref{barfeta}$,
we see that, for $|u-\langle u\rangle_{\varepsilon,\delta}|>\sigma_0$,
\begin{align}\label{etatalyorcc}
\bar{\eta}(u)&=\eta(u)-\eta(\langle u\rangle_{\varepsilon,\delta})
-\eta'(\langle u\rangle_{\varepsilon,\delta}{+0})(u-\langle u\rangle_{\varepsilon,\delta})\nonumber\\[1mm]
&=\int_0^1 \big|\eta'(\langle u\rangle_{\varepsilon,\delta}+\theta(u-\langle u\rangle_{\varepsilon,\delta}))
-\eta'(\langle u\rangle_{\varepsilon,\delta}{+0})\big|{\rm d}\theta\,|u-\langle u\rangle_{\varepsilon,\delta}|\nonumber\\[1mm]
&\geq\int_0^1 \big|\eta'(\langle u\rangle_{\varepsilon,\delta}\pm\theta\sigma_0)
-\eta'(\langle u\rangle_{\varepsilon,\delta}{+0})\big|{\rm d}\theta\,|u-\langle u\rangle_{\varepsilon,\delta}|,
\end{align}
so that
$\langle \bar{\eta}(u)\rangle_{\varepsilon_j,r_{k_i}\varepsilon_j}$ satisfies
\begin{align}\label{etatalyor0}
\langle \bar{\eta}(u)\rangle_{\varepsilon_j,r_{k_i}\varepsilon_j}&\geq c_0\min\Big\{\int_0^1 \big|\eta'(\langle u\rangle_{\varepsilon_j,r_{k_i}\varepsilon_j}\pm\theta\sigma_0)
 -\eta'(\langle u\rangle_{\varepsilon_j,r_{k_i}\varepsilon_j}{+0})\big|{\rm d}\theta\Big\}\nonumber\\[1mm]
&=:P_\eta(\langle u\rangle_{\varepsilon_j,r_{k_i}\varepsilon_j},\sigma_0,c_0).
\end{align}
Combining $\eqref{ftalyor0}$ with $\eqref{etatalyor0}$, we obtain
\begin{equation}\label{fetatalyor}
\langle \bar{\eta}(u)\rangle_{\varepsilon_j,r_{k_i}\varepsilon_j}\langle \bar{f}(u)\rangle_{\varepsilon_j,r_{k_i}\varepsilon_j}
\geq P_f(\langle u\rangle_{\varepsilon_j,r_{k_i}\varepsilon_j},\sigma_0,c_0)\, P_\eta(\langle u\rangle_{\varepsilon_j,r_{k_i}\varepsilon_j},\sigma_0,c_0).
\end{equation}

To pass the limits in $\eqref{fetatalyor}$, by $\eqref{ufuxi}$, we can choose a further subsequence (still denoted)
$(\varepsilon_j,r_{k_i})$ in $\eqref{calulimit000}$ such that
 $\lim_{\varepsilon_j\rightarrow 0} \lim_{r_{k_i}\rightarrow 0}\langle u\rangle_{\varepsilon_j,r_{k_i}\varepsilon_j}=:\bar{u}$ and the limits of $\langle \bar{\eta}(u)\rangle_{\varepsilon_j,r_{k_i}\varepsilon_j}$ and $\langle \bar{f}(u)\rangle_{\varepsilon_j,r_{k_i}\varepsilon_j}$ exist. Since $f'(u)$ and $\eta'(u)$ are both strictly increasing, it follows from
 $\eqref{ftalyorcc}$--$\eqref{fetatalyor}$ that
\begin{align}\label{fetatalyor1}
&\lim\limits_{\varepsilon_j\rightarrow 0} \lim\limits_{r_{k_i}\rightarrow 0}\,\langle \bar{\eta}(u)\rangle_{\varepsilon_j,r_{k_i}\varepsilon_j}\langle \bar{f}(u)\rangle_{\varepsilon_j,r_{k_i}\varepsilon_j}\nonumber\\[1mm]
&\geq P_f(\bar{u},\sigma_0,c_0)\, P_\eta(\bar{u},\sigma_0,c_0)\nonumber\\[1mm]
&=c^2_0\min\Big\{\int_0^1 \big|f'(\bar{u}\pm\theta\sigma_0)-f'(\bar{u}{+0})\big|{\rm d}\theta\Big\}\,\min\Big\{\int_0^1 \big|\eta'(\bar{u}\pm\theta\sigma_0)-\eta'(\bar{u}{+0})\big|{\rm d}\theta\Big\}\nonumber\\[1mm]
&\geq c^2_0\min\big\{\big|f'(\bar{u}\pm\theta_1\sigma_0)-f'(\bar{u}{+0})\big|\big\}\,\min\big\{\big|\eta'(\bar{u}\pm\theta_2\sigma_0)-\eta'(\bar{u}{+0})\big|\big\}\nonumber\\[1mm]
&>0
\end{align}
for some $\theta_1,\theta_2\in(0,1)$.

Restricting to subsequence $(\varepsilon_j, r_{k_i})$, it is clear
that $\eqref{fetatalyor1}$ contradicts to $\eqref{convexBuniform0}$
with $\eqref{calBpg0k}$.
We conclude that $\eqref{calulimit}$ holds, so does $\eqref{caluu00}$.
Up to now, we have proved $\eqref{calff00}$ for the case that $f'(u)$ is strictly increasing.

\smallskip
6. We now prove $\eqref{calff00}$ for the case that  $f'(u)$ is merely nondecreasing, $i.e.$,
\begin{equation}\label{concalff00}
\lim_{\varepsilon\rightarrow 0} \lim_{r_k\rightarrow 0}|\langle f(u)-f(\langle u\rangle_{\varepsilon,r_k\varepsilon})\rangle_{\varepsilon,r_k\varepsilon}|=0.
\end{equation}

We first define $\tilde{f}(u)$ and $\tilde{\eta}(u)$ as
\begin{equation}\label{conbarfeta}
\begin{cases}
\tilde{f}(u):=f(u)-f(\langle u\rangle_{\varepsilon,\delta})
  -f'(\langle u\rangle_{\varepsilon,\delta}{+0})\big(u-\langle u\rangle_{\varepsilon,\delta}\big)\geq 0,\\[1mm]
\tilde{\eta}(u):=\eta(u)-\eta(\langle u\rangle_{\varepsilon,\delta})
 -\eta'(\langle u\rangle_{\varepsilon,\delta}{+0})\big(u-\langle u\rangle_{\varepsilon,\delta}\big)\geq 0.
\end{cases}
\end{equation}
Then, from $\langle u-\langle u\rangle_{\varepsilon,\delta}\rangle_{\varepsilon,\delta}=0$, we have
\begin{align}\label{conbarfetaa}
&\langle f(u)-f(\langle u\rangle_{\varepsilon,\delta})\rangle_{\varepsilon,\delta}\nonumber\\[1mm]
&=\big\langle f(u)-f(\langle u\rangle_{\varepsilon,\delta})
  -f'(\langle u\rangle_{\varepsilon,\delta}{+0})
   (u-\langle u\rangle_{\varepsilon,\delta})\big\rangle_{\varepsilon,\delta}\nonumber\\[1mm]
&=\frac{1}{|\Omega_{\varepsilon,\delta}|}\Big(\int_{\Omega_{\varepsilon,\delta}\bigcap\{u<a^-_{\varepsilon,\delta}\}}
+\int_{\Omega_{\varepsilon,\delta}\bigcap\{a^-_{\varepsilon,\delta}\leq u\leq a^+_{\varepsilon,\delta}\}}
+\int_{\Omega_{\varepsilon,\delta}\bigcap\{u>a^+_{\varepsilon,\delta}\}}\Big)\tilde{f}(u)\,{\rm d}t{\rm d}x\nonumber\\[1mm]
&=:J_-(\varepsilon,\delta)+J_0(\varepsilon,\delta)+J_+(\varepsilon,\delta),
\end{align}
where $a^\pm_{\varepsilon,\delta}$,
by applying Proposition $\ref{ProCon}$ to $u=\langle u\rangle_{\varepsilon,\delta}$,
are determined by
\begin{equation}\label{condegenI}
\big[a^-_{\varepsilon,\delta},a^+_{\varepsilon,\delta}\big]:=I^+(\langle u\rangle_{\varepsilon,\delta})=\big\{v\in \bar{{\Bbb R}}\,:\, f(v)-f(\langle u\rangle_{\varepsilon,\delta})-f'(\langle u\rangle_{\varepsilon,\delta}{+0})(v-\langle u\rangle_{\varepsilon,\delta})=0\big\}.
\end{equation}
Furthermore, by $a^\pm_{\varepsilon,\delta}\in I^+(\langle u\rangle_{\varepsilon,\delta})$,
we have
\begin{equation}\label{condegenI1}
\begin{cases}
f(a^\pm_{\varepsilon,\delta})-f(\langle u\rangle_{\varepsilon,\delta})
-f'(\langle u\rangle_{\varepsilon,\delta}{+0})(a^\pm_{\varepsilon,\delta}-\langle u\rangle_{\varepsilon,\delta})=0,
\\[1mm]
f'(a^-_{\varepsilon,\delta}{+0})
 =f'(a^+_{\varepsilon,\delta}{-0})=f'(\langle u\rangle_{\varepsilon,\delta}{+0}).
\end{cases}
\end{equation}

We prove $\eqref{concalff00}$ by passing the limits
of $J_0(\varepsilon,\delta)$ and $J_\pm(\varepsilon,\delta)$, respectively.
For $J_0(\varepsilon,\delta)$, by $\eqref{conbarfeta}$ and $\eqref{condegenI}$,
$\tilde{f}(u)\equiv 0$ on $I^+(\langle u\rangle_{\varepsilon,\delta})$
so that $J_0(\varepsilon,\delta)\equiv 0$ for any $(\varepsilon,\delta)$.

For $J_\pm(\varepsilon,\delta)$, if $a^-_{\varepsilon,\delta}=-\infty$,
then $\Omega_{\varepsilon,\delta}\bigcap\{u<a^-_{\varepsilon,\delta}\}
=\emptyset$ so that $J_-(\varepsilon,\delta)=0$.
If $a^+_{\varepsilon,\delta}=\infty$,
$\Omega_{\varepsilon,\delta}\bigcap\{u>a^+_{\varepsilon,\delta}\}=\emptyset$
so that $J_+(\varepsilon,\delta)=0$.
Thus, we can restrict sequence $(\varepsilon,\delta)$ such that
$a^-_{\varepsilon,\delta}>-\infty$ and $a^+_{\varepsilon,\delta}<\infty$.
Furthermore, by $\eqref{OmegaB}$, $\Omega_{\varepsilon,\delta}\subset B_{\frac{\delta}{\varepsilon}}(0,0)$
is uniformly bounded in $(\varepsilon,\delta)$ with $\delta\leq\varepsilon$ so that,
from $u\in L^\infty_{\rm loc}$ and $f(u)\in {\rm Lip}_{\rm loc}({\Bbb R})$,
$\eqref{conbarfeta}$
implies that the function sequence $\{\tilde{f}(u)\}_{\varepsilon,\delta}$
is uniformly bounded in $(\varepsilon,\delta)$ with $\delta\leq\varepsilon$.
If there exists a subsequence of $\{a^-_{\varepsilon,\delta}\}$ ($resp.,$ $\{a^+_{\varepsilon,\delta}\}$)
tending to $-\infty$ ($resp.,$ $\infty$), by restricting to this subsequence,
$\Omega_{\varepsilon,\delta}\bigcap\{u<a^-_{\varepsilon,\delta}\}$
($resp.,$ $\Omega_{\varepsilon,\delta}\bigcap\{u>a^+_{\varepsilon,\delta}\}$) tends
to the empty set. Then, by $\eqref{conbarfetaa}$, $J_-(\varepsilon,\delta)$
($resp.,$ $J_+(\varepsilon,\delta)$) tends to $0$.
Thus, we can further restrict sequences $\{a^-_{\varepsilon,\delta}\}$
and $\{a^+_{\varepsilon,\delta}\}$ to be both uniformly bounded.
Thus, in order to prove $\eqref{concalff00}$, it suffices to show that
\begin{equation}\label{conJ}
\lim_{\varepsilon\rightarrow 0} \lim_{r_k\rightarrow 0}\, J_\pm(\varepsilon,r_k\varepsilon)
=\lim_{\varepsilon\rightarrow 0} \lim_{r_k\rightarrow 0}\,\frac{1}{|\Omega_{\varepsilon,r_k\varepsilon}|}\int_{\Omega_{\varepsilon,r_k\varepsilon}\bigcap\{\pm(u-a^\pm_{\varepsilon,r_k\varepsilon})>0\}}\tilde{f}(u)\,{\rm d}t{\rm d}x=0
\end{equation}
with restricting to subsequence $(\varepsilon,r_k)$ such that, for some constants $m$ and $M$,
\begin{equation}\label{conJa}
m\leq a^-_{\varepsilon,r_k\varepsilon}\leq a^+_{\varepsilon,r_k\varepsilon}\leq M.
\end{equation}

On the other hand, since $\tilde{f}(u)\geq 0$ on ${\Bbb R}$,
by $\eqref{conbarfetaa}$, $J_\pm(\varepsilon,r_k\varepsilon)\geq 0$.
Applying  $\eqref{linearB}$--$\eqref{convexB}$ and $\eqref{caletafg0}$
to $\tilde{f}(u)$ and $\tilde{\eta}(u)$ as in $\eqref{conbarfeta}$,
for subsequence $(\varepsilon,r_k)$ satisfying $\eqref{conJa}$, we obtain
\begin{equation}\label{conconvexBuniform0}
0\leq J_\pm(\varepsilon,r_k\varepsilon) \langle \tilde{\eta}(u) \rangle_{\varepsilon,r_k\varepsilon}  \leq \mathcal{B}_{\varepsilon,r_k\varepsilon}(\tilde{f},\tilde{\eta})
=\mathcal{B}_{\varepsilon,r_k\varepsilon}(f,\eta)\lesssim \varrho(\varepsilon,r_k\varepsilon).
\end{equation}

According to $\eqref{conbarfeta}$ and $\eqref{condegenI1}$,
for subsequence $(\varepsilon,r_k)$ satisfying $\eqref{conJa}$, we have
\begin{equation}\label{conbarfeta+}
\tilde{f}(u)=f(u)-f(a^\pm_{\varepsilon,r_k\varepsilon})
-f'(\langle u\rangle_{\varepsilon,\delta}{+0})\big(u-a^\pm_{\varepsilon,r_k\varepsilon}\big)\geq 0.
\end{equation}
Since $f(u)\in {\rm Lip}_{\rm loc}({\Bbb R})$, and $u \in L^\infty_{\rm loc}$ in $\eqref{conbarfetaa}$
is uniformly bounded in $(\varepsilon,r_k)$ with $r_k<1$,
$\eqref{conJa}$ and $\eqref{conbarfeta+}$ imply that,
in order to prove $\eqref{conJ}$, it suffices to show that,
for subsequence $(\varepsilon,r_k)$ satisfying $\eqref{conJa}$,
\begin{equation}\label{conU}
\lim_{\varepsilon\rightarrow 0} \lim_{r_k\rightarrow 0}\,
\frac{1}{|\Omega_{\varepsilon,r_k\varepsilon}|}
\int_{\Omega_{\varepsilon,r_k\varepsilon}\bigcap\{\pm(u-a^\pm_{\varepsilon,r_k\varepsilon})>0\}}
\big|u-a^\pm_{\varepsilon,r_k\varepsilon}\big|\,{\rm d}t{\rm d}x=0.
\end{equation}
In fact, for any fixed $\sigma>0$, we have
\begin{equation}\label{concalulessigma}
\frac{1}{|\Omega_{\varepsilon,r_k\varepsilon}|}
\int_{\Omega_{\varepsilon,r_k\varepsilon}\bigcap\{0<\pm(u-a^\pm_{\varepsilon,r_k\varepsilon})\leq \sigma\}}
\big|u-a^\pm_{\varepsilon,r_k\varepsilon}\big|\,{\rm d}t{\rm d}x
\leq\sigma.
\end{equation}
Thus, if we can prove that,
for any fixed $\sigma>0$,
\begin{equation}\label{concalulimit}
\mathop{\overline{\lim}}\limits_{\varepsilon\rightarrow 0}\mathop{\overline{\lim}}
\limits_{r_k\rightarrow 0}\frac{1}{|\Omega_{\varepsilon,r_k\varepsilon}|}
\int_{\Omega_{\varepsilon,r_k\varepsilon}\bigcap\{\pm(u-a^\pm_{\varepsilon,r_k\varepsilon})>\sigma\}}
\big|u-a^\pm_{\varepsilon,r_k\varepsilon}\big|\,{\rm d}t{\rm d}x=0,
\end{equation}
then it follows from
$\eqref{concalulessigma}$--$\eqref{concalulimit}$ that,
for any fixed $\sigma>0$,
$$
\mathop{\overline{\lim}}\limits_{\varepsilon\rightarrow 0}
\mathop{\overline{\lim}}\limits_{r_k\rightarrow 0}\,\
\frac{1}{|\Omega_{\varepsilon,r_k\varepsilon}|}
\int_{\Omega_{\varepsilon,r_k\varepsilon}\bigcap\{\pm(u-a^\pm_{\varepsilon,r_k\varepsilon})>0\}}
\big|u-a^\pm_{\varepsilon,r_k\varepsilon}\big|\,{\rm d}t{\rm d}x
\leq \sigma,
$$
which yields $\eqref{conU}$ directly.

\smallskip
We now prove $\eqref{concalulimit}$ by contradiction.
If
$\eqref{concalulimit}$ does not hold, then there exist $\sigma^\pm_0>0$ and $c^\pm_0>0$ such that
\begin{equation}\label{concalulimit0}
\mathop{\overline{\lim}}\limits_{\varepsilon\rightarrow 0}\mathop{\overline{\lim}}\limits_{r_k\rightarrow 0}\frac{1}{|\Omega_{\varepsilon,r_k\varepsilon}|}
\int_{\Omega_{\varepsilon,r_k\varepsilon}\bigcap\{\pm(u-a^\pm_{\varepsilon,r_k\varepsilon})>\sigma^\pm_0\}}
\big|u-a^\pm_{\varepsilon,r_k\varepsilon}\big|\,{\rm d}t{\rm d}x\geq 2c^\pm_0>0,
\end{equation}
which is equivalent to saying that there exists a subsequence $(\varepsilon_j,r_{k_i})$ such that,
for small $\varepsilon_j$ and $r_{k_i}$,
\begin{equation}\label{concalulimit000}
\frac{1}{|\Omega_{\varepsilon_j,r_{k_i}\varepsilon_j}|}
\int_{\Omega_{\varepsilon_j,r_{k_i}\varepsilon_j}\bigcap\{\pm(u-a^\pm_{\varepsilon_j,r_{k_i}\varepsilon_j})>\sigma^\pm_0\}}
\big|u-a^\pm_{\varepsilon_j,r_{k_i}\varepsilon_j}\big|\,{\rm d}t{\rm d}x\geq c^\pm_0>0.
\end{equation}

Since $f'(u)$ is nondecreasing, by  Proposition $\ref{ProCon}$ and the definition of $a^\pm_{\varepsilon,\delta}$ in $\eqref{condegenI}$,
\begin{equation}\label{genuinef}
(u-a^\pm_{\varepsilon,\delta})\big(f'(u)-f'(\langle u\rangle_{\varepsilon,\delta}{+0})\big)>0
\qquad \mbox{for $\pm(u-a^\pm_{\varepsilon,\delta})>0$}.
\end{equation}
Then, from $\eqref{condegenI1}$ and $\eqref{conbarfeta+}$, for $\pm(u-a^\pm_{\varepsilon,\delta})>\sigma^\pm_0$,
\begin{align}\label{conftalyorcc}
\tilde{f}(u)&=f(u)-f(a^\pm_{\varepsilon,\delta})
-f'(\langle u\rangle_{\varepsilon,\delta}{+0})(u-a^\pm_{\varepsilon,\delta})\nonumber\\[1mm]
&=\int_0^1 \big|f'(a^\pm_{\varepsilon,\delta}+\theta(u-a^\pm_{\varepsilon,\delta}))
 -f'(\langle u\rangle_{\varepsilon,\delta}{+0})\big|{\rm d}\theta\,|u-a^\pm_{\varepsilon,\delta}|\nonumber\\[1mm]
&\geq\int_0^1 \big|f'(a^\pm_{\varepsilon,\delta}\pm\theta\sigma^\pm_0)
 -f'(\langle u\rangle_{\varepsilon,\delta}{+0})\big|{\rm d}\theta\,|u-a^\pm_{\varepsilon,\delta}|\nonumber\\[1mm]
&=\int_0^1 \big|f'(a^\pm_{\varepsilon,\delta}\pm\theta\sigma^\pm_0)
 -f'(a^\pm_{\varepsilon,\delta}{\mp 0})\big|{\rm d}\theta\,|u-a^\pm_{\varepsilon,\delta}|.
\end{align}
Since  $\eta'(u)$ is strictly increasing,
then $(u-v)(\eta'(u)-\eta'(v{\pm 0}))>0$ for any $u\neq v$.
From $\langle u\rangle_{\varepsilon,\delta}\in I^+(\langle u\rangle_{\varepsilon,\delta})=\big[a^-_{\varepsilon,\delta},a^+_{\varepsilon,\delta}\big]$ and the definition of $\tilde{\eta}(u)$ as in $\eqref{conbarfeta}$, for $\pm(u-a^\pm_{\varepsilon,\delta})>\sigma^\pm_0$,
\begin{align}\label{conetatalyorcc}
\tilde{\eta}(u)&=\eta(u)-\eta(\langle u\rangle_{\varepsilon,\delta})
-\eta'(\langle u\rangle_{\varepsilon,\delta}{+0})(u-\langle u\rangle_{\varepsilon,\delta})\nonumber\\[1mm]
&\geq\eta(u)-\eta(a^\pm_{\varepsilon,\delta})-\eta'(a^\pm_{\varepsilon,\delta}{\pm 0})(u-a^\pm_{\varepsilon,\delta})\nonumber\\[1mm]
&=\int_0^1 \big|\eta'(a^\pm_{\varepsilon,\delta}+\theta(u-a^\pm_{\varepsilon,\delta}))
 -\eta'(a^\pm_{\varepsilon,\delta}{\pm 0})\big|{\rm d}\theta\,|u-a^\pm_{\varepsilon,\delta}|\nonumber\\[1mm]
&\geq\int_0^1 \big|\eta'(a^\pm_{\varepsilon,\delta}\pm\theta\sigma^\pm_0)
 -\eta'(a^\pm_{\varepsilon,\delta}{\pm 0})\big|{\rm d}\theta \,|u-a^\pm_{\varepsilon,\delta}|.
\end{align}

Therefore, from $\eqref{conbarfetaa}$, $\eqref{concalulimit000}$, and $\eqref{conftalyorcc}$, we have
\begin{align}\label{conftalyor0}
J_\pm(\varepsilon_j,r_{k_i}\varepsilon_j)
&=\frac{1}{|\Omega_{\varepsilon_j,r_{k_i}\varepsilon_j}|}\int_{\Omega_{\varepsilon_j,r_{k_i}\varepsilon_j}\bigcap
\{\pm(u-a^\pm_{\varepsilon_j,r_{k_i}\varepsilon_j})>\sigma^\pm_0\}}
  \tilde{f}(u)\,{\rm d}t{\rm d}x\nonumber\\
&\geq\int_0^1 \big|f'(a^\pm_{\varepsilon_j,r_{k_i}\varepsilon_j}\pm\theta\sigma^\pm_0)
  -f'(a^\pm_{\varepsilon_j,r_{k_i}\varepsilon_j}{\mp 0})\big|{\rm d}\theta\nonumber\\
&\qquad\times\frac{1}{|\Omega_{\varepsilon_j,r_{k_i}\varepsilon_j}|}
  \int_{\Omega_{\varepsilon_j,r_{k_i}\varepsilon_j}\bigcap
  \{\pm(u-a^\pm_{\varepsilon_j,r_{k_i}\varepsilon_j})>\sigma^\pm_0\}}
    \big|u-a^\pm_{\varepsilon_j,r_{k_i}\varepsilon_j}\big|\,{\rm d}t{\rm d}x\nonumber\\
&\geq c^\pm_0\int_0^1 \big|f'(a^\pm_{\varepsilon_j,r_{k_i}\varepsilon_j}\pm\theta\sigma^\pm_0)
 -f'(a^\pm_{\varepsilon_j,r_{k_i}\varepsilon_j}{\mp 0})\big|{\rm d}\theta\nonumber\\
&=:K_f(a^\pm_{\varepsilon_j,r_{k_i}\varepsilon_j},\sigma^\pm_0,c^\pm_0).
\end{align}
Similarly, for $\tilde{\eta}(u)$ as in $\eqref{conbarfeta}$,
from $\eqref{concalulimit000}$ and $\eqref{conetatalyorcc}$, we have
\begin{equation}\label{conetatalyor0}
\langle \tilde{\eta}(u)\rangle_{\varepsilon_j,r_{k_i}\varepsilon_j}\geq c^\pm_0\int_0^1 \big|\eta'(a^\pm_{\varepsilon_j,r_{k_i}\varepsilon_j}\pm\theta\sigma^\pm_0)
-\eta'(a^\pm_{\varepsilon_j,r_{k_i}\varepsilon_j}{\pm 0})\big|{\rm d}\theta=:K_\eta(a^\pm_{\varepsilon_j,r_{k_i}\varepsilon_j},\sigma^\pm_0,c^\pm_0).
\end{equation}
Combining $\eqref{conftalyor0}$ with $\eqref{conetatalyor0}$, we obtain
\begin{equation}\label{confetatalyor}
J_\pm(\varepsilon_j,r_{k_i}\varepsilon_j)\langle \tilde{\eta}(u)\rangle_{\varepsilon_j,r_{k_i}\varepsilon_j}
\geq  K_f(a^\pm_{\varepsilon_j,r_{k_i}\varepsilon_j},\sigma^\pm_0,c^\pm_0)\, K_\eta(a^\pm_{\varepsilon_j,r_{k_i}\varepsilon_j},\sigma^\pm_0,c^\pm_0).
\end{equation}

To pass the limits in $\eqref{confetatalyor}$, by $\eqref{ufuxi}$ and $\eqref{conJa}$, we can choose a subsequence (still denoted)
$(\varepsilon_j,r_{k_i})$ in $\eqref{concalulimit000}$ so that
 $\lim_{\varepsilon_j\rightarrow 0} \lim_{r_{k_i}\rightarrow 0}\langle u\rangle_{\varepsilon_j,r_{k_i}\varepsilon_j}=:\tilde{u}$ and $\lim_{\varepsilon_j\rightarrow 0} \lim_{r_{k_i}\rightarrow 0}a^\pm_{\varepsilon_j,r_{k_i}\varepsilon_j}=:a^\pm(\tilde{u})$, and the limits of $J_\pm(\varepsilon_j,r_{k_i}\varepsilon_j)$ and $\langle \tilde{\eta}(u)\rangle_{\varepsilon_j,r_{k_i}\varepsilon_j}$ exist.
 Since $f'(u)$ is nondecreasing and $\eta'(u)$ is strictly increasing,
it follows  from $\eqref{condegenI1}$, $\eqref{genuinef}$, and $\eqref{conftalyor0}$--$\eqref{confetatalyor}$ that
\begin{align}\label{confetatalyor1}
&\lim\limits_{\varepsilon_j\rightarrow 0} \lim\limits_{r_{k_i}\rightarrow 0}\,J_\pm(\varepsilon_j,r_{k_i}\varepsilon_j)\langle \tilde{\eta}(u)\rangle_{\varepsilon_j,r_{k_i}\varepsilon_j}\nonumber\\[1mm]
&\geq K_f(a^\pm(\tilde{u}),\sigma^\pm_0,c^\pm_0)\, K_\eta(a^\pm(\tilde{u}),\sigma^\pm_0,c^\pm_0)\nonumber\\[1mm]
&=(c^\pm_0)^2\int_0^1 \big|f'(a^\pm(\tilde{u})\pm\theta\sigma^\pm_0)
 -f'(a^\pm(\tilde{u}){\mp 0})\big|{\rm d}\theta
\int_0^1 \big|\eta'(a^\pm(\tilde{u})\pm\theta\sigma^\pm_0)
-\eta'(a^\pm(\tilde{u}){\pm 0})\big|{\rm d}\theta\nonumber\\[1mm]
&\geq (c^\pm_0)^2\, \big|f'(a^\pm(\tilde{u})\pm\tilde{\theta}_1\sigma^\pm_0)
-f'(a^\pm(\tilde{u}){\mp 0})\big|\,
 \big|\eta'(a^\pm(\tilde{u})\pm\tilde{\theta}_2\sigma^\pm_0)
    -\eta'(a^\pm(\tilde{u}){\pm 0})\big|\nonumber\\[1mm]
&>0
\end{align}
for some $\tilde{\theta}_1,\tilde{\theta}_2\in(0,1)$.

Restricting to subsequence $(\varepsilon_j, r_{k_i})$, it is clear
that $\eqref{confetatalyor1}$ contradicts to $\eqref{conconvexBuniform0}$
with $\eqref{calBpg0k}$.
We conclude that $\eqref{concalulimit}$ holds, so does $\eqref{conU}$.
Therefore, we have proved $\eqref{calff00}$  for the case that $f'(u)$ is nondecreasing.

This completes the proof of Theorem $\ref{Theinfty}$.

\begin{remark}\label{varphitx0}
To understand $\eqref{equationxi}$, we first see that $\eqref{xifu}$ infers that
$$
\langle (\varphi_{\varepsilon})_t \rangle_{\varepsilon,\delta}
+f(\langle(\varphi_{\varepsilon})_x\rangle_{\varepsilon,\delta})
=-\langle f(u)\rangle_{\varepsilon,\delta}+f(\langle u \rangle_{\varepsilon,\delta})\leq 0,
$$
by applying the Jensen inequality to the convex function $f(u)$.
Passing the limit and noting that
$\varphi_\varepsilon =\varphi-\varepsilon |(t,x)|$ with $\varphi \in C^1$,  we obtain
$$
\big(\varphi_t+f(\varphi_x)\big)(0,0) \leq 0.
$$
On the other hand, a viscosity supersolution of $\eqref{equationh}$ with convex function $f(u)$
should satisfy
$$
\big(\varphi_t+f(\varphi_x)\big)(0,0)\geq 0.
$$
Therefore, \eqref{equationxi} should hold at the minimum point $(t,x)=(0,0)$ for $w-\varphi$.
\end{remark}

\smallskip
\section{Proof of Theorem 2.4}
In Theorem $\ref{Thep}$, we assume that
$u\in L^p_{\rm loc}$  such that $Q(u)\in L^1_{\rm loc}$.
In general, function $w=w(t,x)$ defined by $w_x=u$ and $w_t=-f(u)$ in Theorem $\ref{Thep}$ is continuous, which is well-defined for the viscosity solutions of the Hamilton-Jacobi equation $\eqref{equationh}$. However, the H\"{o}lder continuity of $w$ can be obtained
if the limit for $\eta(u)$ in $\eqref{growthrate}$ holds for $\beta>0$, $i.e.$, the entropy function $\eta(u)$ satisfies that
\begin{equation}\label{growthratefeta}
\lim_{u \rightarrow \pm\infty}\frac{\eta(u)}{|u|^{\beta+1}}\geq M_2>0 \qquad {\rm for \ some\ } \beta>0.
\end{equation}

\subsection{H\"{o}lder continuity of $w$}
If $\eqref{growthratefeta}$ holds,  by the definition of $w$ in Theorem $\ref{Thep}$,
$$
w\in L_{\rm loc}^1({\Bbb R}^+\times{\Bbb R}),
\qquad
w_x=u \in L_{\rm loc}^p({\Bbb R}^+\times{\Bbb R}), \qquad
w_t=-f(u)\in L_{\rm loc}^1({\Bbb R}^+\times{\Bbb R}).
$$
We claim that $w$ is {\it H\"{o}lder continuous} in $(t,x)$
on any compact subset in ${\Bbb R}^+\times{\Bbb R}$.
This can be seen in the following four steps:

\smallskip
1. Testing with a cut-off function $\phi$ in $x$, we obtain from $\eqref{etamu2}$ that
$\int\phi(x)\eta(u(t,x))\,{\rm d}x$ is locally bounded in $t$, {\it i.e.},
\begin{equation}\label{etaloc}
\eta(u)\in L^\infty_{\rm loc}\big({\Bbb R}^+_t,L^1_{\rm loc}({\Bbb R}_x)\big).
\end{equation}
Then, using $\eqref{growthratefeta}$, we have
\begin{equation}\label{uloc}
u\in L^\infty_{\rm loc}\big({\Bbb R}^+_t,L^{\beta+1}_{\rm loc}({\Bbb R}_x)\big).
\end{equation}
Since $w_x=u$, it follows from $\eqref{uloc}$ that, for any $t\in [t_1,t_2]\subset {\Bbb R}^+_t$,
\begin{equation}\label{htcloc}
[w(t,\cdot)]_{C^{0,\gamma_1}_{\rm loc}({\Bbb R}_x)}
\lesssim\|u\|_{L^{\beta+1}_{\rm loc}({\Bbb R}_x)}<\infty
\qquad \mbox{\rm for $\gamma_1=\frac{\beta}{\beta+1}$},
\end{equation}
which implies that $w$ is H\"{o}lder continuous in $x$, locally uniformly in $t$, $i.e.$,
\begin{equation}\label{htxloc}
w\in L^\infty_{\rm loc}\big({\Bbb R}^+_t,C^{0,\gamma_1}_{\rm loc}({\Bbb R}_x)\big)
\qquad \mbox{\rm for $\gamma_1=\frac{\beta}{\beta+1}$},
\end{equation}
where ${\Bbb R}^+_t$ and ${\Bbb R}_x$
represent the corresponding domains for the independent variables $t$ and $x$, respectively.

\smallskip
2. We now show that
\begin{equation}\label{htxcloc}
w\in C^{0,\frac{1}{\gamma}}_{\rm loc}\big({\Bbb R}^+_t,L^1_{\rm loc}({\Bbb R}_x)\big),
\end{equation}
where $\gamma$ is given by $\eqref{quadratic2}$.

Fix $t_2>t_1$ and $x_2>x_1$. Since $w_t=-f(u)$,
for any $t,t'\in [t_1,t_2]$ with $t'>t$,
\begin{equation}\label{ht2t1}
\Big|\int_{x_1}^{x_2}|w(t',\xi)|\,{\rm d}\xi-\int_{x_1}^{x_2}|w(t,\xi)|\,{\rm d}\xi\Big|
\leq\int_{x_1}^{x_2}|w(t',\xi)-w(t,\xi)|\,{\rm d}\xi
\leq\int_{t}^{t'}\int_{x_1}^{x_2}|f(u)|\,{\rm d}\tau{\rm d}\xi.
\end{equation}

For the case that  $\eqref{quadratic2}$ holds for $\gamma=1$,
$|f(u)|\lesssim |\eta(u)|+1$ for $u\in {\Bbb R}$ so that, by $\eqref{etaloc}$,
\begin{align}\label{alphalesbeta}
\int_{t}^{t'}\int_{x_1}^{x_2}|f(u)|\,{\rm d}\tau{\rm d}\xi
&\lesssim \int_{t}^{t'}\int_{x_1}^{x_2}(|\eta(u)|+1)\,{\rm d}\tau{\rm d}\xi \nonumber\\[1mm]
&\leq \max_{t\in [t_1,t_2]}\big\{\|\eta(u(t,\cdot))\|_{L^{1}([x_1,x_2])}\big\}\,(t'-t)+(x_2-x_1)(t'-t)\nonumber\\[1mm]
&\lesssim t'-t.
\end{align}

For the case that $\eqref{quadratic2}$ holds for some $\gamma>1$, but does not hold for $\gamma=1$,
by a simple calculation, $\eqref{quadratic2}$ implies that
$$
|f(u)|\,(|\eta(u)|+1)^{-\frac{1}{\gamma}}\lesssim (1+|Q(u)|)^{\frac{\gamma-1}{\gamma}}
\qquad \mbox{\rm for $u\in {\Bbb R}$}.
$$
Since $u\in L^p_{\rm loc}$ such that $Q(u)\in L^1_{\rm loc}$, by the Sobolev inequality, we obtain
\begin{align}\label{alphalesbeta1}
&\int_{t}^{t'}\int_{x_1}^{x_2}|f(u)|\,{\rm d}\tau{\rm d}\xi\nonumber\\[1mm]
&=\int_{t}^{t'}\int_{x_1}^{x_2}(|\eta(u)|+1)^{\frac{1}{\gamma}}\, \big(|f(u)|\,(|\eta(u)|+1)^{-\frac{1}{\gamma}}\big)\,{\rm d}\tau{\rm d}\xi\nonumber\\[1mm]
&\leq\Big(\int_{t}^{t'}\int_{x_1}^{x_2}(|\eta(u)|+1)\,{\rm d}\tau{\rm d}\xi\Big)^{\frac{1}{\gamma}}\,
\Big(\int_{t}^{t'}\int_{x_1}^{x_2}\big(|f(u)|\,(|\eta(u)|+1)^{-\frac{1}{\gamma}}\big)^{\frac{\gamma}{\gamma-1}}\,{\rm d}\tau{\rm d}\xi\Big)^{\frac{\gamma-1}{\gamma}}\nonumber\\[1mm]
&\lesssim \Big(\int_{t}^{t'}\int_{x_1}^{x_2}(|\eta(u)|+1)\,{\rm d}\tau{\rm d}\xi\Big)^{\frac{1}{\gamma}}\,
\Big(\int_{t}^{t'}\int_{x_1}^{x_2}(1+|Q(u)|)
 \,{\rm d}\tau{\rm d}\xi\Big)^{\frac{\gamma-1}{\gamma}}\nonumber\\[1mm]
&\leq \Big(\int_{t}^{t'}\int_{x_1}^{x_2}(|\eta(u)|+1)\,{\rm d}\tau{\rm d}\xi\Big)^{\frac{1}{\gamma}}\,
\Big(\int_{t_1}^{t_2}\int_{x_1}^{x_2}(1+|Q(u)|)\,{\rm d}\tau{\rm d}\xi\Big)^{\frac{\gamma-1}{\gamma}}\nonumber\\[1mm]
&\lesssim \Big(\int_{t}^{t'}\int_{x_1}^{x_2}(|\eta(u)|+1)\,{\rm d}\tau{\rm d}\xi\Big)^{\frac{1}{\gamma}}
\lesssim (t'-t)^{\frac{1}{\gamma}}.
\end{align}
Combining $\eqref{ht2t1}$--$\eqref{alphalesbeta1}$ together, we obtain $\eqref{htxcloc}$.

\smallskip
3. By interpolation, \eqref{htxloc}--\eqref{htxcloc} imply that
\begin{equation}\label{hxtloc}
w\in L^\infty_{\rm loc}\big({\Bbb R}_x,C^{0,\gamma_2}_{\rm loc}({\Bbb R}^+_t)\big)
\qquad \mbox{\rm for $\gamma_2=\frac{\beta}{\gamma(2\beta+1)}$}.
\end{equation}

In fact, let $\zeta\in C^\infty_{\rm c}({\Bbb R}_x)$ be non-negative
with ${\rm spt}\,\zeta(x)\subset(-1,1)$
and $\int_{{\Bbb R}}\zeta(x)\,{\rm d}x=1$.
Set $\zeta_\varepsilon(x):=\frac{1}{\varepsilon}\zeta(\frac{x}{\varepsilon})$,
and let $*$ denote the convolution in the $x-$variable. Then we have
\begin{align}\label{hxtloc1}
|w(t',x)-w(t,x)|
&\leq|w(t',x)-(w*\zeta_\varepsilon)(t',x)|+|(w*\zeta_\varepsilon)(t',x)-(w*\zeta_\varepsilon)(t,x)|\nonumber\\[1mm]
&\quad\,+|w(t,x)-(w*\zeta_\varepsilon)(t,x)|\nonumber\\[1mm]
&\leq \varepsilon^{\gamma_1}\sup_{y,|y-x|\leq\varepsilon}\frac{|w(t',x)-w(t',y)|}{|x-y|^{\gamma_1}}
+\sup_{{\Bbb R}}|\zeta|\,\,\frac{1}{\varepsilon}
   \int_{x-\varepsilon}^{x+\varepsilon}|w(t',y)-w(t,y)|\,{\rm d}y\nonumber\\[1mm]
&\quad\,+\varepsilon^{\gamma_1}\sup_{y,|y-x|\leq\varepsilon}\frac{|w(t,x)-w(t,y)|}{|x-y|^{\gamma_1}}\nonumber\\[1mm]
&\lesssim\varepsilon^{\gamma_1}+\varepsilon^{-1}|t'-t|^{\frac{1}{\gamma}}.
\end{align}
Choosing $\varepsilon=|t'-t|^\sigma$ with $\sigma=\frac{\beta+1}{\gamma(2\beta+1)}$,
then $\gamma_1\sigma=\frac{1}{\gamma}-\sigma=\gamma_2$. Thus, $\eqref{hxtloc1}$ implies $\eqref{hxtloc}$.

\smallskip
4. For any bounded $\mathcal{K}\Subset {\Bbb R}^+\times{\Bbb R}$, we have
\begin{equation}\label{hxtcontinuity}
\sup_{(t,x),(t',x')\in \mathcal{K}}\frac{|w(t',x')-w(t,x)|}{|t'-t|^{\gamma_2}+|x'-x|^{\gamma_1}}<\infty,
\end{equation}
where $\gamma_1$ in $\eqref{htxloc}$ and $\gamma_2$ in $\eqref{hxtloc}$ are determined
by constants $\beta>0$ and $\gamma\geq 1$, which are respectively given by $\eqref{growthratefeta}$ and $\eqref{quadratic2}$.

In fact, according to $\eqref{htxloc}$ and $\eqref{hxtloc}$,
for any $(t,x),(t',x')\in \mathcal{K}\subset[t_1,t_2]\times[x_1,x_2]$,
\begin{align*}
\frac{|w(t',x')-w(t,x)|}{|t'-t|^{\gamma_2}+|x'-x|^{\gamma_1}}&
\leq\frac{|w(t',x')-w(t',x)|}{|t'-t|^{\gamma_2}+|x'-x|^{\gamma_1}}
+\frac{|w(t',x)-w(t,x)|}{|t'-t|^{\gamma_2}+|x'-x|^{\gamma_1}}\\[1mm]
&\leq\frac{|w(t',x')-w(t',x)|}{|x'-x|^{\gamma_1}}+\frac{|w(t',x)-w(t,x)|}{|t'-t|^{\gamma_2}}\\[1mm]
&\leq {\rm ess}\sup_{t'\in[t_1,t_2]}[w(t',\cdot)]_{C^{0,\gamma_1}([x_1,x_2])}
+{\rm ess}\sup_{x\in[x_1,x_2]}[w(\cdot,x)]_{C^{0,\gamma_2}([t_1,t_2])}\\[1mm]
&<\infty.
\end{align*}

\begin{remark}\label{Holderab}
For the case that $f(u)\simeq M_1|u|^{\alpha+1}$ and $\eta(u)\simeq M_2|u|^{\beta+1}$ as $u \to \pm\infty$ for some $\alpha, \beta>0$.
If $\alpha\leq \beta$, then $\eqref{quadratic2}$ holds for $\gamma=1${\rm ;} and if $\alpha> \beta$,  then $\eqref{quadratic2}$ holds for $\gamma=\frac{\alpha+1}{\beta+1}$.
Therefore, $\eqref{hxtcontinuity}$ holds for $\gamma_1=\frac{\beta}{\beta+1}$ and $\gamma_2$ given by
$$
\gamma_2=\frac{\beta}{2\beta+1}\quad \mbox{\rm if $\alpha\leq\beta$}, \qquad\,\,\,
\gamma_2=\frac{\beta(\beta+1)}{(\alpha+1)(2\beta+1)}\quad \mbox{\rm if $\alpha>\beta$}.
$$
\end{remark}

\subsection{Viscosity solutions}
Similar to \S 4.1, it can be checked that $w$ is a viscosity subsolution of $\eqref{equationh}$.

\smallskip
Now we show that $w$ is also a viscosity supersolution of $\eqref{equationh}$.
That is to say, if $\varphi$ is a $C^1$--function such that $w-\varphi$ has a minimum at $(0,0)$ with $(w-\varphi)(0,0)=0$,
then
\begin{equation}\label{equationxip}
\big(\varphi_t+f(\varphi_x)\big)(0,0)=0.
\end{equation}

Following the same arguments as Step 2 in \S 4.2,
we let $\varphi_\varepsilon=\varphi-\varepsilon|(t,x)|$ with $\varepsilon\in(0,1]$
such that $w-\varphi_\varepsilon$ has a strict minimum at $(0,0)$.
Let $\Omega_{\varepsilon,\delta}$
and $\langle\,\cdot\,\rangle_{\varepsilon,\delta}$ be
defined by $\eqref{defOmega}$ and $\eqref{integralu}$, respectively.
Then $\eqref{OmegaB}$ and $\eqref{xifu}$--$\eqref{hxi}$ all hold.

We divide the remaining proof into three steps.

\smallskip
1. According to $\eqref{calBinfty}$, in order to complete the estimate
of $\mathcal{B}_{\varepsilon,\delta}(f,\eta)$ as Step 4
in \S 4.2, we need the estimates of $I_1(\varepsilon,\delta)$ in $\eqref{calIinfty}$ and $I_2(\varepsilon,\delta)$ in $\eqref{calIIinfty}$.

For $I_1(\varepsilon,\delta)$, it suffices to show the last step of $\eqref{calIinfty}$,
which is directly implied by
\begin{equation}\label{deltaomegap}
\delta^2\lesssim |\Omega_{\varepsilon,\delta}| \qquad \mbox{for\ small $\delta>0$}.
\end{equation}
By the same arguments for $\Omega_{\varepsilon,\delta}$ as Step 4 in \S 4.2,
the key point for proving $\eqref{deltaomegap}$ is to establish $\eqref{hxithxix}$.
For this purpose, it suffices to show that
\begin{equation}\label{uflesssim0}
\langle |u| \rangle_{\varepsilon,\delta}\lesssim 1, \qquad
\langle |f(u)| \rangle_{\varepsilon,\delta}\lesssim 1.
\end{equation}
Since $|\langle (\varphi_{\varepsilon})_t\rangle_{\varepsilon,\delta}|\lesssim 1$ and
$|\langle (\varphi_{\varepsilon})_x\rangle_{\varepsilon,\delta}|\lesssim 1$, it
follows from $\eqref{xifu}$ that
\begin{equation}\label{uflesssim}
|\langle u \rangle_{\varepsilon,\delta}|\lesssim 1, \qquad |\langle f(u)\rangle_{\varepsilon,\delta}|\lesssim 1.
\end{equation}
According to $\eqref{growthrate}$ and $\eqref{uflesssim}$, we have
\begin{equation}\label{ualphalesssim}
\langle |u| \rangle _{\varepsilon,\delta}\lesssim 1+ \langle|f(u)| \rangle_{\varepsilon,\delta} \lesssim 1+ |\langle f(u)\rangle_{\varepsilon,\delta}|\lesssim 1,
\end{equation}
which yields $\eqref{uflesssim0}$ so that $\eqref{hxithxix}$ holds.
Thus, by the same argument for $\eqref{calIinfty}$, it follows from $\eqref{etamu2}$ and $\eqref{deltaomegap}$ that
\begin{equation}\label{calIp}
I_1(\varepsilon,\delta)\leq \frac{1}{|\Omega_{\varepsilon,\delta}|}\int \max\{\delta-(w-\varphi_\varepsilon),0\}\,{\rm d}\mu\leq \frac{\delta}{|\Omega_{\varepsilon,\delta}|}\mu(\Omega_{\varepsilon,\delta})\lesssim \frac{1}{\delta}\mu(B_{\frac{\delta}{\varepsilon}}(0,0)).
\end{equation}

For $I_2(\varepsilon,\delta)$, it suffices to establish the additional estimates of $\langle| \eta(u)|\rangle_{\varepsilon,r_k\varepsilon}$ and $\langle |q(u)|\rangle_{\varepsilon,r_k\varepsilon}$, where $\{r_k\}$ is the subsequence in $\eqref{mu0k}$.  By $\eqref{growthrate}$--$\eqref{quadratic1}$,
we first have
\begin{equation}\label{grdfeta}
\langle| \eta(u)|\rangle_{\varepsilon,r_k\varepsilon} \lesssim 1+ \langle|q(u)|\rangle_{\varepsilon,r_k\varepsilon},\qquad \langle|q(u)|\rangle_{\varepsilon,r_k\varepsilon} \lesssim 1+ \langle Q(u)\rangle_{\varepsilon,r_k\varepsilon},
\end{equation}
and
\begin{equation}\label{grdfetaq}
\lim_{u\rightarrow \pm\infty} h(u):=\lim_{u\rightarrow \pm\infty} \frac{Q(u)}{|q(u)|}=\infty.
\end{equation}
According to the definition of $I_1(\varepsilon,\delta)$ in $\eqref{calBinfty}$, from $\eqref{calIp}$--$\eqref{grdfetaq}$, we have
\begin{align}
\langle h(u)|q(u)|\rangle_{\varepsilon,r_k\varepsilon}
&=\left\langle (-f(u),u)\cdot(\eta (u),q(u))\right\rangle_{\varepsilon,r_k\varepsilon} \nonumber\\[1mm]
&=\left\langle((\varphi_{\varepsilon})_t, (\varphi_{\varepsilon})_x)\cdot(\eta(u),q(u))\right\rangle_{\varepsilon,r_k\varepsilon}+I_1(\varepsilon,r_k\varepsilon)\nonumber\\[1mm]
&\textstyle\lesssim \sup_{\Omega_{\varepsilon,r_k\varepsilon}}
  \left |((\varphi_{\varepsilon})_t, (\varphi_{\varepsilon})_x)\right|\left\langle\left |(\eta(u),q(u))\right|\right\rangle_{\varepsilon,r_k\varepsilon}+\frac{1}{r_k}\mu(B_{r_k}(0,0))\varepsilon^{-1}\nonumber\\[1mm]
&\textstyle\lesssim \langle|\eta(u)|\rangle_{\varepsilon,r_k\varepsilon} +\langle|q(u)|\rangle_{\varepsilon,r_k\varepsilon} + \frac{1}{r_k}\mu(B_{r_k}(0,0))\varepsilon^{-1}\nonumber\\[1mm]
&\textstyle\lesssim
\langle |q(u)|\rangle_{\varepsilon,r_k\varepsilon}+1+\rho_k\, \varepsilon^{-1},\label{calup}
\end{align}
where $\rho_k:=\frac{1}{r_k}\mu(B_{r_k}(0,0))$.
This, by $\eqref{deflesssim}$, implies that
\begin{equation}\label{calupab}
\langle h(u)|q(u)|\rangle_{\varepsilon,r_k\varepsilon}\leq C\big(\langle |q(u)|\rangle_{\varepsilon,r_k\varepsilon}+1+\rho_k\, \varepsilon^{-1}\big)
\end{equation}
for some constant $C>0$. From $\eqref{grdfetaq}$, we have
\begin{equation}\label{growthrateetaQ}
\langle Q(u)\rangle_{\varepsilon,r_k\varepsilon}=\langle h(u)|q(u)|\rangle_{\varepsilon,r_k\varepsilon}\lesssim 1+\langle (h(u)-C)|q(u)|\rangle_{\varepsilon,r_k\varepsilon}
\lesssim 1+\rho_k\, \varepsilon^{-1},
\end{equation}
which, by $\eqref{grdfeta}$, implies that
\begin{equation}\label{growthrateetaq}
\langle|q(u)|\rangle_{\varepsilon,r_k\varepsilon}
\lesssim 1+\rho_k\, \varepsilon^{-1}, \qquad \langle| \eta(u)|\rangle_{\varepsilon,r_k\varepsilon} \lesssim 1+\rho_k\, \varepsilon^{-1}.
\end{equation}
Therefore, the estimate of $I_2(\varepsilon,\delta)$ in $\eqref{calIIinfty}$ becomes
\begin{align}\label{calIIp}
I_2(\varepsilon,r_k\varepsilon)
&\textstyle=\sup_{\Omega_{\varepsilon,r_k\varepsilon}}\big|((\varphi_{\varepsilon})_t, (\varphi_{\varepsilon})_x)
 - \left\langle((\varphi_{\varepsilon})_t, (\varphi_{\varepsilon})_x)\right\rangle_{\varepsilon,r_k\varepsilon}\big|\,
 \left\langle\left |(\eta(u),q(u)) \right|\right \rangle _{\varepsilon,r_k\varepsilon}\nonumber\\[1mm]
&\lesssim \big(r_k+\varepsilon\big) \big(\langle|\eta(u)|\rangle_{\varepsilon,r_k\varepsilon} + \langle|q(u)|\rangle_{\varepsilon,r_k\varepsilon}\big)\nonumber\\[1mm]
&\lesssim (r_k+\varepsilon)(1+\rho_k\,\varepsilon^{-1})\nonumber\\[1mm]
&=\varepsilon(1+r_k\,\varepsilon^{-1})(1+\rho_k\,\varepsilon^{-1}).
\end{align}

Using $\eqref{calIp}$ and $\eqref{calIIp}$, similar to $\eqref{calBinfty}$, we have
\begin{align}\label{calBp}
\mathcal{B}_{\varepsilon,r_k\varepsilon}(f,\eta)&\leq I_1(\varepsilon,r_k\varepsilon)+I_2(\varepsilon,r_k\varepsilon)\nonumber\\[1mm]
&\lesssim \rho_k\,\varepsilon^{-1}+\varepsilon(1+r_k\,\varepsilon^{-1})(1+\rho_k\,\varepsilon^{-1})
=: \tilde{\varrho}(\varepsilon,r_k).
\end{align}
Then it follows from $\eqref{convexB}$ that
\begin{equation}\label{caletafg}
0\leq \langle \eta(u)-\eta(\langle u\rangle_{\varepsilon,r_k\varepsilon})\rangle_{\varepsilon,r_k\varepsilon}\,\langle f(u)-f(\langle
u\rangle_{\varepsilon,r_k\varepsilon})\rangle_{\varepsilon,r_k\varepsilon}
\leq \mathcal{B}_{\varepsilon,r_k\varepsilon}(f,\eta)\lesssim \tilde{\varrho}(\varepsilon,r_k).
\end{equation}
On the other hand, by the definition of $\tilde{\varrho}(\varepsilon,r_k)$ in $\eqref{calBp}$, from $\eqref{mu0k}$ and $\eqref{calup}$, we have
\begin{equation}\label{calBpg}
\lim_{\varepsilon\rightarrow 0} \lim_{r_k\rightarrow 0}\tilde{\varrho}(\varepsilon,r_k)=0.
\end{equation}

\smallskip
2. We now show that
\begin{equation}\label{calff0}
\lim_{\varepsilon\rightarrow 0} \lim_{r_k\rightarrow 0}|\langle f(u)-f(\langle u\rangle_{\varepsilon,r_k\varepsilon})\rangle_{\varepsilon,r_k\varepsilon}|=0.
\end{equation}

Following the same arguments as Step 6 in \S 4.2, we choose $\tilde{f}(u)$ and $\tilde{\eta}(u)$ as in $\eqref{conbarfeta}$, $i.e.$,
\begin{equation}\label{conbarfetalp}
\begin{cases}
\tilde{f}(u):=f(u)-f(\langle u\rangle_{\varepsilon,\delta})
  -f'(\langle u\rangle_{\varepsilon,\delta}{+0})(u-\langle u\rangle_{\varepsilon,\delta})\geq 0,\\[1mm]
\tilde{\eta}(u):=\eta(u)-\eta(\langle u\rangle_{\varepsilon,\delta})
 -\eta'(\langle u\rangle_{\varepsilon,\delta}{+0})(u-\langle u\rangle_{\varepsilon,\delta})\geq 0.
\end{cases}
\end{equation}
Since $\eta(u)$ is strictly convex as in $\eqref{convexg}$, then $\tilde{\eta}(u)>0$ for $u\neq \langle u\rangle_{\varepsilon,\delta}$. By $\eqref{uflesssim}$, there exists a $N_1>0$ such that $|\langle u\rangle_{\varepsilon,\delta}|\leq N_1$ holds uniformly for $(\varepsilon,\delta)$ with $\delta\leq\varepsilon$.
Then $I^+(N_1)$ as in $\eqref{degenIc}$ is bounded;
otherwise, $f(u)$ would be a linear function so that $Q(u)=O(1)\eta(u)$ for $u>N_1$ or $u<N_1$, which contradicts to $\eqref{quadratic1}$.
 Therefore, we have
\begin{equation}\label{fetatilde0}
\tilde{f}(u)>0,\quad \tilde{\eta}(u)>0  \qquad \mbox{\rm for\ $|u|>N_2$},
\end{equation}
where $N_2:= \max \{ |\sup I^+(N_1)|,\,|\inf I^+(N_1)|\}$.

Furthermore, by $\eqref{growthrate}$ and $\eqref{uflesssim}$, for sufficiently large $N>N_2$,
\begin{equation}\label{fetatilde}
\tilde{f}(u)\sim f(u)-f(\langle u\rangle_{\varepsilon,\delta}),\quad \tilde{\eta}(u)\sim \eta(u)-\eta(\langle u\rangle_{\varepsilon,\delta})\qquad \mbox{\rm for $|u|>N$}.
\end{equation}
Noticing that $\langle u-\langle u\rangle_{\varepsilon,\delta}\rangle_{\varepsilon,\delta}=0$, from $\eqref{conbarfetalp}$, we have
\begin{align}\label{conbarfetaap}
\langle f(u)-f(\langle u\rangle_{\varepsilon,\delta})\rangle_{\varepsilon,\delta}
&=\big\langle f(u)-f(\langle u\rangle_{\varepsilon,\delta})
-f'(\langle u\rangle_{\varepsilon,\delta}{+0})\big(u-\langle u\rangle_{\varepsilon,\delta}\big)\big\rangle_{\varepsilon,\delta}\nonumber\\[1mm]
&=\frac{1}{|\Omega_{\varepsilon,\delta}|}\Big(\int_{\Omega^-_{\varepsilon,\delta}}
+\int_{\Omega^0_{\varepsilon,\delta}}
+\int_{\Omega^+_{\varepsilon,\delta}}
+\int_{\Omega^N_{\varepsilon,\delta}}\Big)\tilde{f}(u)\,{\rm d}t{\rm d}x\nonumber\\[1mm]
&=:J_-(\varepsilon,\delta)+J_0(\varepsilon,\delta)+J_+(\varepsilon,\delta)+J_N(\varepsilon,\delta),
\end{align}
where $a^\pm_{\varepsilon,\delta}$ are given by $\eqref{condegenI}$, and $\Omega^\pm_{\varepsilon,\delta}$, $\Omega^0_{\varepsilon,\delta}$, and $\Omega^N_{\varepsilon,\delta}$ are defined by
\begin{equation}\label{conomegap}
\begin{cases}
\Omega^-_{\varepsilon,\delta}=\Omega_{\varepsilon,\delta}\bigcap\{-N<u<a^-_{\varepsilon,\delta}\},\quad
&\Omega^0_{\varepsilon,\delta}=\Omega_{\varepsilon,\delta}\bigcap\{a^-_{\varepsilon,\delta}\leq u\leq a^+_{\varepsilon,\delta}\}, \\[1mm]
\Omega^+_{\varepsilon,\delta}=\Omega_{\varepsilon,\delta}\bigcap\{a^+_{\varepsilon,\delta}<u<N\},\quad
&\Omega^N_{\varepsilon,\delta}=\Omega_{\varepsilon,\delta}\bigcap\{|u|\geq N\}.
\end{cases}
\end{equation}

For $J_0(\varepsilon,\delta)$, it follows from $\eqref{conbarfeta}$ and $\eqref{condegenI}$
that $\tilde{f}(u)\equiv 0$ on $I^+(\langle u\rangle_{\varepsilon,\delta})$, so that
\begin{equation}\label{conJp0}
J_0(\varepsilon,\delta)\equiv 0 \qquad {\rm for\ any}\ (\varepsilon,\delta).
\end{equation}

For $J_\pm(\varepsilon,\delta)$, from $|\langle u \rangle_{\varepsilon,\delta}|\lesssim 1$
in $\eqref{uflesssim}$ and
$u$ on $\Omega^\pm_{\varepsilon,\delta}$
%$u\in \Omega^\pm_{\varepsilon,\delta}$
uniformly bounded by $N$,
by the same arguments as in Step 6 in \S 4.2, we have
\begin{equation}\label{conJp}
\lim_{\varepsilon\rightarrow 0} \lim_{r_k\rightarrow 0}\, J_\pm(\varepsilon,r_k\varepsilon)=\lim_{\varepsilon\rightarrow 0} \lim_{r_k\rightarrow 0}\,\frac{1}{|\Omega_{\varepsilon,r_k\varepsilon}|}\int_{\Omega^\pm_{\varepsilon,r_k\varepsilon}}\tilde{f}(u)\,{\rm d}t{\rm d}x=0.
\end{equation}

For $J_N(\varepsilon,\delta)$, we now show that
\begin{equation}\label{conJpN}
\lim_{\varepsilon\rightarrow 0} \lim_{r_k\rightarrow 0}\, J_N(\varepsilon,r_k\varepsilon)=\lim_{\varepsilon\rightarrow 0} \lim_{r_k\rightarrow 0}\,\frac{1}{|\Omega_{\varepsilon,r_k\varepsilon}|}\int_{\Omega^N_{\varepsilon,r_k\varepsilon}}\tilde{f}(u)\,{\rm d}t{\rm d}x=0.
\end{equation}
Applying  $\eqref{linearB}$--$\eqref{convexB}$ and $\eqref{caletafg}$ to $\tilde{f}(u)$ and $\tilde{\eta}(u)$,
we obtain from $\eqref{conbarfetaap}$  that
\begin{equation}\label{conconvexBuniformp}
0\leq J_N(\varepsilon,r_k\varepsilon) \langle \tilde{\eta}(u) \rangle_{\varepsilon,r_k\varepsilon}  \leq \mathcal{B}_{\varepsilon,r_k\varepsilon}(\tilde{f},\tilde{\eta})
=\mathcal{B}_{\varepsilon,r_k\varepsilon}(f,\eta)\lesssim \tilde{\varrho}(\varepsilon,r_k).
\end{equation}
Denote $U_N(\varepsilon,\delta)$ by
\begin{equation}\label{defualphabeta}
U_N(\varepsilon,\delta):=\frac{1}{|\Omega_{\varepsilon,\delta}|}\int_{\Omega^N_{\varepsilon,\delta}}\tilde{\eta}(u)\,{\rm d}t{\rm d}x.
\end{equation}

For the case that $\eqref{quadratic2}$ holds for $\gamma=1$, $f(u) \lesssim \eta(u)$ as $u \to \pm\infty$ so that
\begin{equation}\label{ualphabeta}
J_N(\varepsilon,\delta)=\frac{1}{|\Omega_{\varepsilon,\delta}|}\int_{\Omega^N_{\varepsilon,\delta}}\tilde{f}(u)\,{\rm d}t{\rm d}x
\lesssim \frac{1}{|\Omega_{\varepsilon,\delta}|}\int_{\Omega^N_{\varepsilon,\delta}}\tilde{\eta}(u)\,{\rm d}t{\rm d}x=U_N(\varepsilon,\delta).
\end{equation}
Noticing that $\tilde{f}(u),\tilde{\eta}(u)\geq 0$ on ${\Bbb R}$, from $\eqref{conbarfetaap}$ and $\eqref{conconvexBuniformp}$--$\eqref{ualphabeta}$, we have
\begin{equation}\label{ualphabeta1}
0\leq \big(J_N(\varepsilon,r_k\varepsilon)\big)^2\lesssim J_N(\varepsilon,r_k\varepsilon)U_N(\varepsilon,r_k\varepsilon) \leq
J_N(\varepsilon,r_k\varepsilon) \langle \tilde{\eta}(u) \rangle_{\varepsilon,r_k\varepsilon}\lesssim \tilde{\varrho}(\varepsilon,r_k),
\end{equation}
which, by $\eqref{mu0k}$ and $\eqref{calBpg}$, implies that $\eqref{conJpN}$ holds.

For the case that $\eqref{quadratic2}$ holds for some $\gamma>1$, but does not hold for $\gamma=1$. We claim:
\begin{equation}\label{quadratic2tilde}
\frac{\tilde{Q}(u)}{\tilde{\eta}(u)} \lesssim \Big(\frac{\tilde{Q}(u)}{\tilde{f}(u)}\Big)^{\gamma} \qquad {\rm as\ } u\rightarrow \pm\infty,
\end{equation}
where functions $\tilde{Q}(u)$ and $\tilde{q}(u)$ are defined by
\begin{equation}\label{Qqtilde}
\tilde{Q}(u):=(u-\langle u\rangle_{\varepsilon,\delta})(\tilde{q}(u)-\tilde{q}(\langle u\rangle_{\varepsilon,\delta}))-\tilde{f}(u)\tilde{\eta}(u),\qquad \tilde{q}(u):=\int_0^u\tilde{\eta}'(\xi)\tilde{f}'(\xi)\, {\rm d}\xi.
\end{equation}
By the definition of $\tilde{f}(u)$ and $\tilde{\eta}(u)$ in $\eqref{conbarfetalp}$,  as $u \to \pm\infty$,
\begin{equation}\label{qtilde}
\tilde{q}(u)=\int_0^u\big(\eta'(\xi)-\eta'(\langle u\rangle_{\varepsilon,\delta}{+0})\big)\big(f'(\xi)-f'(\langle u\rangle_{\varepsilon,\delta}{+0})\big)\, {\rm d}\xi =O(1) q(u).
\end{equation}
By $\eqref{grdfetaq}$ and $\eqref{Qqtilde}$--$\eqref{qtilde}$, as $u \to \pm\infty$,
\begin{align}\label{Qtilde}
\tilde{Q}(u)&=(u-\langle u\rangle_{\varepsilon,\delta})(\tilde{q}(u)-\tilde{q}(\langle u\rangle_{\varepsilon,\delta}))-\tilde{f}(u)\tilde{\eta}(u)\nonumber\\[1mm]
 &=(u-\langle u\rangle_{\varepsilon,\delta})(q(u)-q(\langle u\rangle_{\varepsilon,\delta}))-(f(u)-f(\langle u\rangle_{\varepsilon,\delta}))(\eta(u)-\eta(\langle u\rangle_{\varepsilon,\delta}))\nonumber\\[1mm]
&=uq(u)-f(u)\eta(u)+O(1)q(u)\nonumber\\[1mm]
&=h(u)|q(u)|+O(1)q(u) \nonumber\\[1mm]
&=(h(u)+O(1))|q(u)|\simeq Q(u).
\end{align}
Combining \eqref{fetatilde0}--\eqref{fetatilde} with $\eqref{Qtilde}$,  $\eqref{quadratic2}$ implies that $\eqref{quadratic2tilde}$ holds.

Since $u\in L^p_{\rm loc}$ such that $Q(u)\in L^1_{\rm loc}$, and $\Omega_{\varepsilon,\delta}\subset B_{\frac{\delta}{\varepsilon}}(0,0)$
is uniformly bounded in $(\varepsilon, \delta)$ with $\delta\leq\varepsilon$, from $\eqref{fetatilde0}$, $\eqref{quadratic2tilde}$, and $\eqref{Qtilde}$,
by the Sobolev inequality, we have
\begin{align}\label{ualphabetaaa}
J_N(\varepsilon,\delta)&=\frac{1}{|\Omega_{\varepsilon,\delta}|}\int_{\Omega^N_{\varepsilon,\delta}}\tilde{f}(u)\,{\rm d}t{\rm d}x\nonumber\\[1mm]
&=\frac{1}{|\Omega_{\varepsilon,\delta}|}\int_{\Omega^N_{\varepsilon,\delta}}\tilde{\eta}(u)^{\frac{1}{\gamma}}\, \big(\tilde{f}(u)\tilde{\eta}(u)^{-\frac{1}{\gamma}}\big)\,{\rm d}t{\rm d}x\nonumber\\[1mm]
&\leq \Big(\frac{1}{|\Omega_{\varepsilon,\delta}|}\int_{\Omega^N_{\varepsilon,\delta}}\tilde{\eta}(u)\,{\rm d}t {\rm d}x\Big)^{\frac{1}{\gamma}}\, \Big(\frac{1}{|\Omega_{\varepsilon,\delta}|}\int_{\Omega^N_{\varepsilon,\delta}}\big(\tilde{f}(u)\tilde{\eta}(u)^{-\frac{1}{\gamma}}\big)^{\frac{\gamma}{\gamma-1}}\,{\rm d}t{\rm d}x\Big)^{\frac{\gamma-1}{\gamma}}\nonumber\\[1mm]
&\leq \big(U_N(\varepsilon,\delta)\big)^{\frac{1}{\gamma}}\, \Big(\big\langle\big(\tilde{f}(u)\tilde{\eta}(u)^{-\frac{1}{\gamma}}\big)^{\frac{\gamma}{\gamma-1}}
\big\rangle_{\varepsilon,\delta}\Big)^{\frac{\gamma-1}{\gamma}}\nonumber\\[1mm]
&\lesssim \big(U_N(\varepsilon,\delta)\big)^{\frac{1}{\gamma}}\, \big(\langle \tilde{Q}(u)
\rangle_{\varepsilon,\delta}\big)^{\frac{\gamma-1}{\gamma}}\nonumber\\[1mm]
&\lesssim \big(U_N(\varepsilon,\delta)\big)^{\frac{1}{\gamma}}\, \big(\langle Q(u)
\rangle_{\varepsilon,\delta}\big)^{\frac{\gamma-1}{\gamma}}.
\end{align}
Noticing that $\tilde{\eta}(u)\geq 0$ on ${\Bbb R}$, it follows from $\eqref{growthrateetaQ}$  and $\eqref{conconvexBuniformp}$ that
\begin{align}\label{ualphabeta2}
0\leq \big(J_N(\varepsilon,r_k\varepsilon)\big)^{\gamma+1}
 &\lesssim J_N(\varepsilon,r_k\varepsilon)\,U_N(\varepsilon,r_k\varepsilon)\,\big(\langle Q(u) \rangle_{\varepsilon,\delta}\big)^{\gamma-1}\nonumber\\[1mm]
 & \lesssim
J_N(\varepsilon,r_k\varepsilon)\, \langle \tilde{\eta}(u) \rangle_{\varepsilon,r_k\varepsilon}\,(1+\rho_k\,\varepsilon^{-1})^{\gamma-1}\nonumber\\[1mm]
&\lesssim \tilde{\varrho}(\varepsilon,r_k)\,(1+\rho_k\,\varepsilon^{-1})^{\gamma-1}.
\end{align}
which, by $\eqref{mu0k}$ and $\eqref{calBpg}$, implies that $\eqref{conJpN}$ holds.

Therefore, from $\eqref{conbarfetaap}$--$\eqref{conomegap}$,
$\eqref{conJp0}$--$\eqref{conJpN}$ directly imply $\eqref{calff0}$.

\smallskip
3. From $\eqref{OmegaB}$, $\eqref{xifu}$, and $\eqref{mu0k}$, we conclude
\begin{align}\label{calxip}
\big|\big(\varphi_t+f(\varphi_x)\big)(0,0)\big|&\lesssim\big|\langle \varphi_t\rangle_{\varepsilon,r_k\varepsilon}+ f(\langle\varphi_x\rangle_{\varepsilon,r_k\varepsilon})\big|+r_k\nonumber\\[1mm]
&\lesssim\big|\langle (\varphi_{\varepsilon})_t\rangle_{\varepsilon,r_k\varepsilon}
+ f(\langle(\varphi_{\varepsilon})_x\rangle_{\varepsilon,r_k\varepsilon})\big|+\varepsilon+r_k\nonumber\\[1mm]
&=\big|-\langle f(u)\rangle_{\varepsilon,r_k\varepsilon}+ f(\langle u\rangle_{\varepsilon,r_k\varepsilon})\big|+\varepsilon+r_k\nonumber\\[1mm]
&=\big|\langle f(u)-f(\langle u\rangle_{\varepsilon,r_k\varepsilon})\rangle_{\varepsilon,r_k\varepsilon}\big|+\varepsilon+r_k.
\end{align}
Using $\eqref{calff0}$,
and letting first $r_k$ and then $\varepsilon$ go to $0$ in $\eqref{calxip}$, we obtain
$$
\big(\varphi_t+f(\varphi_x)\big)(0,0)=0,
$$
as desired, which means that $w$ is also a viscosity supersolution.

This completes the proof of Theorem $\ref{Thep}$.

\smallskip
\begin{remark}\label{QF}
From the argument in {\rm \S 5.2},
we know that condition $\eqref{quadratic1}$ is used to formulate $\eqref{growthrateetaQ}$ and hence $\eqref{growthrateetaq}$ so that the essential inequality $\eqref{caletafg}$ holds{\rm ;} otherwise, the inequality in $\eqref{calup}$ is trivial so that we can not obtain $\eqref{growthrateetaQ}$, $etc.$.

Condition $\eqref{quadratic2}$ is used to formulate $\eqref{ualphabetaaa}$ so that, by using $\eqref{growthrateetaQ}$, $\eqref{conconvexBuniformp}$
can imply $\eqref{ualphabeta2}$. To formulate $\eqref{ualphabeta2}$ from $\eqref{conconvexBuniformp}$, we need a positive lower bound of $\langle \tilde{\eta}(u) \rangle_{\varepsilon,r_k\varepsilon}$. Since $\langle \tilde{\eta}(u) \rangle_{\varepsilon,r_k\varepsilon}$ can also tend to zero,
we could not expect that there exists a sufficiently small positive constant to be the lower bound of $\langle \tilde{\eta}(u) \rangle_{\varepsilon,r_k\varepsilon}$.
Therefore, we need to require $\langle \tilde{\eta}(u) \rangle_{\varepsilon,r_k\varepsilon}$ bounded by $\langle \tilde{f}(u) \rangle_{\varepsilon,r_k\varepsilon}$
from below in some sense, which is exactly what $\eqref{quadratic2}$ means.
\end{remark}

\bibliographystyle{amsplain}
%\bigskip

\end{document}